\documentclass{article}
\usepackage[english]{babel}
\usepackage[utf8]{inputenc}
\usepackage{bm}
\usepackage{amsmath,amsthm,amsfonts}
\usepackage{graphicx}
\usepackage{color}
\usepackage{multirow}
\usepackage{subcaption}

\newtheorem{theorem}{Theorem}

\newcommand{\const}{\mathop{\rm const}\nolimits}

\graphicspath{ {./figs/} }

\date{}

\title{Efficient decoupling schemes for multiscale multicontinuum problems in fractured porous media}

\author{
Maria Vasilyeva 
\thanks{Department of Mathematics and Statistics, Texas A\&M University  - Corpus Christi,   Corpus Christi, Texas, USA. Email: {\tt maria.vasilyeva@tamucc.edu}.}
}

\begin{document}

\maketitle

\begin{abstract}
We consider the coupled system of equations that describe flow in fractured porous media. To describe such types of problems, multicontinuum and multiscale approaches are used. 
Because in multicontinuum models, the permeability of each continuum has a significant difference, a large number of iterations is required for the solution of the resulting linear system of equations at each time iteration. The presented decoupling technique separates equations for each continuum that can be solved separately, leading to a more efficient computational algorithm with smaller systems and faster solutions. This approach is based on the additive representation of the operator with semi-implicit approximation by time, where the continuum coupling part is taken from the previous time layer. We apply, analyze and numerically investigate decoupled schemes for classical multicontinuum problems in fractured porous media on sufficiently fine grids with finite volume approximation. We show that the decoupled schemes are stable, accurate, and computationally efficient. 
Next, we extend and investigate this approach for multiscale approximation on the coarse grid using the nonlocal multicontinuum (NLMC) method. In NLMC approximation, we construct similar decoupled schemes with the same continuum separation approach. A numerical investigation is presented for model problems with two and three-continuum in the two-dimensional formulation.
\end{abstract}

\section{Introduction}

Numerical simulation of the flow processes in fractured porous media plays an essential role in reservoir simulation ($CO_2$ sequestration, unconventional gas production, and geothermal energy production).
Fractures usually have complex geometries,  high permeability, multiple scales, and very small thicknesses compared to typical reservoir sizes. 
A typical model approach for fractured media is based on the lower-dimensional representation of the fracture objects   \cite{martin2005modeling, d2012mixed, formaggia2014reduced, Quarteroni2008coupling, schwenck2015dimensionally}.
In such models, we have a coupled mixed dimensional system of equations with $d$ - dimensional equation for flow in a porous matrix and $(d-1)$ - dimensional equation for fracture networks with the cross-flow between them.

Approximation techniques of mixed dimensional models can be classified based on the applied meshing method. The discrete fracture model (DFM) is associated with the conforming discretization, or explicit meshing of the fracture geometry \cite{hoteit2008efficient, karimi2003efficient, karimi2001numerical, garipov2016discrete}.
This method is computationally expensive since a large number of unknowns arise in discrete problems. However, the DFM is an accurate tool to describe the flow in fractured porous media.
In another approach, the fractures are not resolved by grid but are considered as an overlaying continuum (embedded fracture model, EFM) \cite{hkj12, ctene2016algebraic, tene2016multiscale}. In EFM, matrix and fracture are viewed as two porosity/permeability types co-existing at one spatial location. The concept of this approach can be classified in the class of dual-continuum or multicontinuum models \cite{barenblatt1960basic, warren1963behavior, douglas1990dual, ginting2011application}. We note that both approaches are accurate for sufficiently fine grids.  

Due to the heterogeneity of the properties and multiple scales, the simulation of the flow processes in fractured porous media requires a very fine grid for accurate approximation, which is computationally expensive. 
The multiscale methods or upscaling techniques are used to reduce the dimension of the fine-grid system \cite{houwu97, eh09, weinan2007heterogeneous, lunati2006multiscale, jenny2005adaptive}.
In previous works, we presented the coarse-scale model based on the Generalized Multiscale Finite Element Method (GMsFEM) for flow in fractured porous media \cite{akkutlu2015multiscale, chung2017coupling, efendiev2015hierarchical, akkutlu2018multiscale}. The general idea of GMsFEM is to design suitable spectral problems to describe important flow modes or continua in basis construction \cite{EGG_MultiscaleMOR, egh12, chung2016adaptive, CELV2015}. Recently, the authors in \cite{chung2017constraint} proposed a constraint energy minimization GMsFE method (CEM-GMsFEM). In CEM-GMsFEM,  the multiscale basis functions are constructed to capture long channelized effects in an oversampling domain. Using such multiscale basis functions, we recently presented a nonlocal multicontinuum (NLMC) method \cite{chung2017non} for problems in heterogeneous fractured media. We remark that since the multiscale basis functions are computed in an oversampled domain, the transfers between fractures and matrix become nonlocal. In the NLMC approximation, the resulting system is very similar to the traditional finite volume method but very accurate due to the nonlocal coupling in each continuum and between them.

In time approximation, additive operator-difference schemes are a helpful tool for solving unsteady equations. Such schemes are constructed for different evolutionary vector problems or systems of equations, for example, Navier-Stokes equations, poro- and thermo-elasticity problems, etc. (see \cite{vabishchevich2013additive} for details). The computational algorithm based on explicit-implicit approximations in time is presented in \cite{gaspar2014explicit}. 
In \cite{kolesov2014splitting}, we consider the coupled systems of linear unsteady partial differential equations, which arise in the modeling of poroelasticity processes. We constructed several splitting schemes and gave some numerical comparisons for typical poroelasticity problems. 
We observe that it is a promising technique for solving large three-dimensional coupled problems, where variable separation can provide a more efficient algorithm and more straightforward implementation. Application of the splitting schemes with loose coupling to solve the fluid-structure interaction problems in hemodynamics is presented in \cite{bukavc2016operator}, where splitting is based on a time discretization by an operator-splitting scheme. In our previous work, we numerically investigated splitting schemes for the thermoporoelasticity problem in fractured media \cite{ammosov2019splitting}. We presented and investigated splitting schemes for multicontinuum media with fixed stress splitting to effectively solve the coupled system of equations for pressures, temperatures, and displacements. Recently, the application of the additive schemes has been considered in  \cite{efendiev2021temporal,efendiev2021splitting}. 
The partially-explicit time discretizations for nonlinear multiscale problems is considered in \cite{chung2021contrast}. The method was combined with the machine learning techniques applied for the implicit part of the operator \cite{efendiev2022efficient}. 
The next extension of the partially explicit scheme is presented in 
\cite{leung2022multirate}, where the multirate partially explicit scheme for multiscale flow problems are considered and analyzed.

In this work, we consider a class of multicontinuum problems that are classically used to describe flow in fractured porous media \cite{barenblatt1960basic, warren1963behavior}. The mathematical model is described by a coupled system of equations for pressure in each continuum. 
To approximate by space, we use a regular finite volume approximation with an embedded fracture model to simulate large thin fractures (hydraulic fractures) and a dual continuum (dual porosity dual permeability) approach for natural fractures. Such models are similar and can be described using the multicontinuum approach \cite{vasilyeva2019nonlocal}. 
The main goal of the work is to develop, analyze and investigate an efficient decoupling scheme that allows separating equations for each continuum. The approach is based on the semi-implicit approximation by the time where the coupling part between each continuum is taken from the previous time layer. We apply and analyze this approach for classical multicontinuum problems in fractured porous media on regular, sufficiently fine grids. We show that the presented scheme is stable, accurate, and computationally efficient. Presented decoupled schemes separate the continuum, and the resulting equation in a sequence requires a smaller number of iterations.
Next, we consider a multiscale approximation, where the same multicontinuum approach can be used to construct an accurate reduced-order model by space. We construct the multicontinuum upscaled models based on the nonlocal multicontinuum (NLMC) method, where multiscale basis functions are calculated in a local domain for each continuum. Numerical results show that the coupled scheme for the NLMC  method provides an accurate and efficient upscaled model on the coarse grid. The presented continuum decoupling approach can be naturally extended for the multiscale nonlocal multicontinuum models and inherit properties of the regular finite volume approximation. This work presents a numerical investigation of the two-dimensional multicontinuum problem. However, we expect a good efficiency for three-dimensional problems, where the number of unknowns is larger, which will be considered in future works. 

The paper is organized as follows.  In Section 2, we describe problem formulation for flow in multicontinuum media.  
Section 3 presents a finite volume approximation by space for multicontinuum problems in fractured porous media with coupled and decoupled schemes.  The extension of the method for multiscale approximation is given in Section 4.  In Section 5, we present numerical results for two test problems in two-dimensional formulation (two- and three-continuum models).  A numerical investigation is presented for regular finite volume approximation on the fine grid and nonlocal multicontinua approximation on the coarse grid.  A conclusion is drawn in Section 6.

\section{Problem formulation}

Dual porosity models describe flow in naturally fractured media and are common in reservoir simulation \cite{barenblatt1960basic, warren1963behavior}.  The mathematical model is described by a coupled system of equations for flow in a porous matrix and natural fracture continuum
\begin{equation}
\label{mm-dp}
\begin{split}
& c_1 \frac{ \partial p_1 }{\partial t} 
- \nabla \cdot (k_1 \nabla p_1) +  \sigma_{12} (p_1 - p_2) =   f_1, 
\quad  x \in \Omega, \\
& c_2 \frac{ \partial p_2 }{\partial t} 
- \nabla \cdot (k_2 \nabla p_2) +   \sigma_{12} (p_2 - p_1)=   f_2.
\quad  x \in \Omega,
\end{split}
\end{equation}
where $\Omega \in \mathcal{R}^d$ is the domain where the porous matrix and natural fracture coexist and $d$ is the dimension. 
Here 
$p_1$ is the pressure in the porous matrix continuum, 
$p_2$ is the pressure in the natural fractures continuum, 
$c_1$, $c_2$, $k_1$ and $k_2$ are the problem coefficients, 
$\sigma_{12} $ are the coupling terms between continuum, 
$f_1$ and $f_2$ are the source terms.  
This model is a macroscale model based on the idealization of the naturally fractured media and can be derived based on homogenization \cite{douglas1990dual}.  

For large-scale fractures (for example, hydraulic fractures), models with explicit fracture representation are usually used. A common approach for flow simulation is based on the mixed dimensional formulation \cite{martin2005modeling, d2012mixed, formaggia2014reduced, Quarteroni2008coupling, schwenck2015dimensionally}.  Let  $\gamma \in \mathcal{R}^{d-1}$ be the lower dimensional domain for fractures.  Then the mathematical model is described by the following coupled system of equations for pressure in hydraulic fracture continuum  ($p_f$) and porous matrix ($p_m$):
\begin{equation}
\label{mm-md}
\begin{split}
& c_m \frac{ \partial p_m }{\partial t} 
- \nabla \cdot (k_m \nabla p_m) +  \sigma_{mf} (p_m - p_f) =   f_m, 
\quad  x \in \Omega, \\
& c_f \frac{ \partial p_f }{\partial t} 
- \nabla \cdot (k_f \nabla p_f) +  \sigma_{mf} (p_f - p_m)=   f_f.
\quad  x \in \gamma.
\end{split}
\end{equation}

We can notice that models are similar and can be generalized for the multicontinuum case for the more complex flow phenomena where both natural and hydraulic fractures coexist (three-continuum model)
\begin{equation}
\label{mm-tc}
\begin{split}
& c_{1} \frac{ \partial  p_1}{\partial t}
- \nabla \cdot ( k_{1}  \nabla p_{1})
+   \sigma_{1 2} (p_1 - p_2)
+   \sigma_{1 f} (p_1 - p_f) = f_1, \quad
x \in \Omega, \\
& c_{2} \frac{ \partial p_2}{\partial t}
- \nabla \cdot (k_{2}  \nabla p_{2})
+   \sigma_{21} (p_2 - p_1)
+   \sigma_{2f} (p_2 - p_f)  = f_2, \quad
x \in \Omega, \\
& c_{f} \frac{ \partial  p_f}{\partial t}
- \nabla \cdot (k_{f}  \nabla p_{f})
+   \sigma_{f1} (p_f - p_1)
+   \sigma_{f2} (p_f - p_2)  = f_f, \quad
x \in \gamma,
\end{split}
\end{equation}
where $\sigma_{\alpha \beta} = \sigma_{\beta \alpha}$  is the coupling coefficient between continuum $\alpha$ and $\beta$ that characterize a flow between them. 

We can generalize the model for the multicontinuum case
\begin{equation}
\label{mm}
c_{\alpha} \frac{ \partial  p_{\alpha}}{\partial t}
- \nabla \cdot  (k_{\alpha}  \nabla p_{\alpha})
+  \sum_{\beta \neq \alpha} \sigma_{\alpha \beta} (p_{\alpha} - p_{\beta})  = f_{\alpha},  \quad x \in \Omega_{\alpha},
\end{equation}
where $\alpha, \beta = 1,2,...,L$ and $L$ is the number of continuum.  

The system of equations \eqref{mm} is considered with the homogeneous Neumann boundary conditions for each continuum
\begin{equation}
\label{mm-bc}
- k_{\alpha}  \nabla p_{\alpha} \cdot n = 0, \quad x \in \partial \Omega_{\alpha},
\end{equation}
where $n$ is the outer normal vector to the domain boundary $\partial \Omega_{\alpha}$, 
and some given initial conditions
\[
p_{\alpha}(t)= p_{\alpha, 0}, \quad t=0,
\] 
in $\Omega_{\alpha}$.

The multicontinuum models are widely used in reservoir simulations. For example,  in gas production from shale formation, we have a highly heterogeneous and complex mixture of organic matter, inorganic matter, and multiscale fractures \cite{akkutlu2012multiscale, akkutlu2017multiscale}. Another example is the fractured vuggy reservoirs, where multicontinuum models are used to characterize the complex interaction between vugges, fractures, and porous matrix \cite{wu2011multiple, wu2007triple, yao2010discrete}.

In this paper, we follow the approach presented in \cite{vabishchevich2013additive, kolesov2014splitting,gaspar2014explicit} and represent the system of equations \eqref{mm}  as a system of equations for vector $p = (p_1, p_2, ...p_L)^T$. 
Let $p \in V$ and $V = V_1 \oplus V_2 \oplus ... \oplus V_L$  be a direct sum of spaces $V_{\alpha}$, where $p_{\alpha} \in V_{\alpha}$ and  $V_{\alpha}$ is the Hilbert space.   Therefore for $p(t) \in V$, we have the following system of equations
\begin{equation}
\label{mm1}
\mathcal{C} \frac{dp}{dt} + \mathcal{A} p = f(t),   \quad 0 < t  \leq T,
\end{equation}
with 
\[
\mathcal{A} = 
\mathcal{D} + \mathcal{Q}, 
\quad
\mathcal{D} =
 \begin{pmatrix}
\mathcal{D}_1 & 0 & ... & 0\\
0 & \mathcal{D}_2 & ... & 0\\
... & ... & ... & ...\\
0 & 0 & ... & \mathcal{D}_L
\end{pmatrix}, 
\quad 
\mathcal{C} = 
\begin{pmatrix}
c_1 & 0 & ... & 0\\
0 & c_2 & ... & 0\\
... & ... & ... & ...\\
0 & 0 & ... & c_L
\end{pmatrix}, 
\]
and 
\[
\mathcal{Q} =  
 \begin{pmatrix}
\sum_{\beta \neq 1} \sigma _{1\beta} & -\sigma_{12} & ... & -\sigma_{1L}\\
-\sigma_{21} & \sum_{\beta \neq 2} \sigma_{2\beta} & ... & -\sigma_{2L}\\
... & ... & ... & ...\\
-\sigma_{L1} & -\sigma_{L2} & ... &  \sum_{\beta \neq L} \sigma_{L\beta}\\
\end{pmatrix}, 
\quad 
f = \begin{pmatrix}
f_1\\
f_2\\ 
...\\
f_L
\end{pmatrix}, 
\]
where $\mathcal{D}_{\alpha} p_{\alpha} = - \nabla \cdot (k_{\alpha}   \nabla p_{\alpha})$ is the diffusion operator for component $\alpha$.  
We consider \eqref{mm1} with initial condition
\begin{equation}
\label{mm1-ic}
p(t) = p_0, \quad t = 0.
\end{equation}
with $p_0 = (p_{10}, p_{20}, ..., p_{L0})^T$ and $p_{\alpha0}$ is the initial condition for component $\alpha$. 

For the system of equations \eqref{mm1} that describe the flow problems in multicontinuum media,  we have the following physical properties
\[
c_{\alpha} \geq 0, \quad 
k_{\alpha} \geq 0, \quad 
\sigma_{\alpha \beta} = \sigma_{\beta \alpha}. 
\]
Therefore  $\mathcal{A}$ and $\mathcal{C}$  are self-adjoint and positive definite operators.  
Let $(p,v)$ and $(p,v)_{\mathcal{A}} = (\mathcal{A}p, v)$ are the scalar products for $p,v \in V$, and $||p|| = \sqrt{(p,p)}$ and $||p||_{\mathcal{A}} = \sqrt{ (\mathcal{A}p, v)}$ be the norms in $V$ and $V_{\mathcal{A}}$. 
For the problem \eqref{mm1} with initial condition \eqref{mm1-ic}, we have the following a priory estimate of the stability of the solution with respect to the initial condition and right-hand side. 

\begin{theorem}
\label{t:t1}
The solution of the problem \eqref{mm1} satisfies the following a priory estimate
\begin{equation}
\label{t1}
||p(t)||_{\mathcal{A}}^2 \leq ||p(0)||_{\mathcal{A}}^2 
+ \frac{1}{2} \int_0^t ||f(s)||^2_{\mathcal{C}^{-1}} ds.
\end{equation}
\end{theorem}
\begin{proof}
By multiplying the equation \eqref{mm1} by $\frac{dp}{dt}$, we obtain
\[
\left(\mathcal{C} \frac{dp}{dt}, \frac{dp}{dt}\right) 
+ \left( \mathcal{A} p,   \frac{dp}{dt} \right) = \left(f, \frac{dp}{dt} \right). 
\]
Next, using integration by parts and free boundary conditions \eqref{mm-bc}
\[
\frac{1}{2} \frac{d}{dt} \left( \mathcal{A} p,  p\right) =  \left( \mathcal{A} p,   \frac{dp}{dt} \right),
\]
and applying a Cauchy inequality 
for the right-hand side
\[
\left(f, \frac{dp}{dt} \right) \leq 
\left| \left( f, \frac{dp}{dt} \right) \right| \leq 
\left\| \frac{dp}{dt} \right\|_{\mathcal{C}}^2 + \frac{1}{4} \left\| f \right\|_{\mathcal{C}^{-1}}^2,
\]
we obtain the following estimate
\[
\left(\mathcal{C} \frac{dp}{dt}, \frac{dp}{dt}\right) 
+ \frac{1}{2} \frac{d}{dt} \left( \mathcal{A} p,  p \right) 
\leq 
\left\| \frac{dp}{dt} \right\|_{\mathcal{C}}^2 + \frac{1}{4} \left\| f \right\|_{\mathcal{C}^{-1}}^2,
\]
or
\[
\frac{d}{dt}  \left\| p \right\|_{\mathcal{A}}^2
\leq 
\frac{1}{2} \left\| f \right\|_{\mathcal{C}^{-1}}^2.
\]
Finally,  after integration by time, we obtain inequality \eqref{t1} . 
\end{proof}

Next, we consider regular approximation by space variables on the sufficiently fine grid.  We use a finite volume method with an embedded fracture model for mixed dimensional flow problems.  

\section{Decoupling schemes for the fine grid finite volume approximation}

To approximate by space, we use a regular finite volume approximation with an embedded fracture model. 
For time approximation, we first consider a regular implicit scheme that leads to the coupled system of equations for multicontinuum media. 
We present the construction of the decoupled schemes that separate equations for each continuum, leading to a more efficient computational algorithm with smaller systems and faster solutions.

\subsection{Finite volume approximation by space}

Let $\Omega$ be the two-dimensional square domain and $\gamma$ be the lower dimensional domain of high-permeable thin fractures.  
In domain $\Omega$, we construct a structured grid $\mathcal{T}_h = \cup_i \varsigma_i$, where $\varsigma_i$ is the square $h \times h$ cell.  
In one-dimensional domain $\gamma$, we construct mesh $\gamma_h = \cup_l \iota_l$. 
For the problem \eqref{mm-md} using a two-point flux approximation, we obtain the following semi-discrete form for $(p_m, p_f)$
\begin{equation}
\label{app-md}
\begin{split}
& c_{m,i} \frac{\partial p_{m, i} }{\partial t}|\varsigma_i|
 + \sum_{j}  T_{m,ij}  (p_{m, i} - p_{m, j})
 + \sum_{n} \sigma_{in} (p_{m, i} - p_{f, n} )
 =  f_{m,i}   |\varsigma_i|, \quad \forall i = 1, N^{m}_h \\
& c_{f,l} \frac{\partial p_{f, l} }{\partial t}  |\iota_l|
+ \sum_{n} T_{f,ln} (p_{f, l} - p_{f, n})
+ \sum_l \sigma_{jl} (p_{f, l} - p_{m, j} )
 =  f_{f,l}  |\iota_l|, \quad \forall l = 1, N^{f}_h
\end{split}
\end{equation}
where $p_{\alpha} = (p_{\alpha, 1}, p_{\alpha, 2},..., p_{\alpha, N^{\alpha}_H})^T$ for $\alpha = m,f$ with $N^m_h = N^{\Omega}_h$ and $N^f_h = N^{\gamma}_h$ ($N^{\Omega}_h$ and $N^{\gamma}_h$ are the number of cells for grid  $\mathcal{T}_h$ and $\gamma_h$).
Here
$T_{m,ij} = k_m |E_{ij}|/d_{ij}$ ($|E_{ij}|$ is the length of facet between cells $\varsigma_i$ and $\varsigma_j$, $d_{ij}$ is the distance between midpoint of cells $\varsigma_i$ and $\varsigma_j$),
$T_{f,ln} = k_f |E_{ln}|/d_{ln}$ ($|E_{ln}|$ and $d_{ln}$ are the same quantities on grind $\gamma_h$), 
$\sigma_{il} = \sigma |E^i_l|/d^i_l$ if $\iota_l \cap \varsigma_i \neq 0$ and zero else ($|E^i_l|$ is the length of fracture-matrix interface and $d^i_l$ is the distance between fracture and matrix of cells).

For the flow problem in three continuum model \eqref{mm-tc} with $ \Omega_1 = \Omega_2 = \Omega$ and $\Omega_3 = \gamma$,  we have similar semi-discrete form for $(p_1, p_2, p_f)$
\begin{equation}
\label{app-tc}
\begin{split}
& c_{1,i} \frac{\partial p_{1, i} }{\partial t} |\varsigma_i|
 + \sum_{j}  T_{1,ij}  (p_{1, i} - p_{1, j})
 +  \sigma_{12,ii} (p_{1, i} - p_{2, i} )
 +  \sum_n \sigma_{1f,in} (p_{1, i} - p_{f, n} )
 =  f_{1,i}   |\varsigma_i|, \quad \forall i = 1, N^1_h \\
 & c_{2,i} \frac{\partial p_{2, i} }{\partial t} |\varsigma_i|
 + \sum_{j}  T_{2,ij}  (p_{2, i} - p_{2, j})
 +  \sigma_{12,ii} (p_{2, i} - p_{1, i} )
 +  \sum_n \sigma_{2f,in} (p_{2, i} - p_{f, n} )
 =  f_{2,i}   |\varsigma_i|, \quad \forall i = 1, N^2_h \\
& c_{f,l} \frac{\partial p_{f,l} }{\partial t} |\iota_l|
+ \sum_{n} T_{f,ln} (p_{f, l} - p_{f, n})
+ \sum_j \sigma_{1f,jl} (p_{f, l} - p_{1, j} )
+ \sum_j \sigma_{2f, jl} (p_{f, l} - p_{2, j} )
 =  f_{f,l}  |\iota_l|, \quad \forall l = 1, N^f_h
\end{split}
\end{equation}
where 
$p_{\alpha} = (p_{\alpha, 1}, p_{\alpha, 2},..., p_{\alpha, N^{\alpha}_H})^T$ for $\alpha=1,2,f$ with  $N^1_h = N^2_h = N^{\Omega}_h$ and $N^f_h = N^{\gamma}_h$, 
$T_{1,ij} = k_1 |E_{ij}|/d_{ij}$, $T_{2,ij} = k_2 |E_{ij}|/d_{ij}$.

The general form for multicontinuum flow problems can be written as follows
\begin{equation}
\label{app}
 c_{\alpha,i} \frac{\partial p_{{\alpha}, i} }{\partial t} |\varsigma^{\alpha}_i|
 + \sum_{j}  T_{{\alpha},ij}  (p_{{\alpha}, i} - p_{{\alpha}, j})
 + \sum_{{\alpha \neq \beta}} 
 \sum_{j} \sigma_{\alpha \beta,ij} (p_{{\alpha}, i} - p_{\beta, j} )
 =  f_{\alpha,i}   |\varsigma^{\alpha}_i|, \quad \forall i = 1, N^{\alpha}_h,
\end{equation}
where $\alpha = 1,2,...,L$.

In the matrix form, we have the following system of coupled equations for $p = (p_1, p_2, ...,p_L)$
\begin{equation}
\label{app-mc}
M \frac{\partial p}{\partial t} + A p = F,
\quad A = D+ Q, 
\end{equation}
with
\[
M = \begin{pmatrix}
M_1 & 0 & ... & 0 \\
0 & M_2 & ... &  0 \\
 ... &  ... &  ...  &   ...   \\
0 & 0 &  ... & M_L
\end{pmatrix},  \quad
D = \begin{pmatrix}
D_1 & 0 & ... & 0 \\
0 & D_2 & ... &  0 \\
 ... &  ... &  ...  &   ...   \\
0 & 0 &  ... & D_L
\end{pmatrix},
\]\[
Q = \begin{pmatrix}
\sum_{\beta \neq 1} Q_{1\beta} & -Q_{12} & ...  & -Q_{1L} \\
-Q_{21} & \sum_{\beta \neq 2} Q_{2\beta} & ... & -Q_{2L} \\
 ... &  ... &  ...  &   ... &  \\
-Q_{L1}   & -Q_{L2} & ... & \sum_{\beta \neq L} Q_{L \beta}
\end{pmatrix},\quad
F = \begin{pmatrix}
F_1 \\
F_2 \\
 ...   \\
F_L
\end{pmatrix},  
\]
where 
$M_{\alpha} = \{m_{\alpha,ij}\}$, 
$D_{\alpha} = \{a_{\alpha,ij}\}$, 
$Q_{\alpha \beta} = \{q_{\alpha \beta, ij}\}$, 
$F_{\alpha} = \{f_{\alpha,j} |\varsigma^{\alpha}_j|\}$
\[
m_{\alpha,ij} =
\left\{\begin{matrix}
c_{\alpha,i} |\varsigma^{\alpha}_i|  & i = j, \\
0 & otherwise
\end{matrix}\right. ,  
\quad
a_{\alpha,ij} =
\left\{\begin{matrix}
\sum_{n \neq i} T_{\alpha, in} & i = j, \\
-T_{\alpha, ij} & otherwise
\end{matrix}\right. ,  
 \quad
q_{\alpha \beta,ij} =
\left\{\begin{matrix}
\sigma_{\alpha \beta,ij} &  \varsigma^{\alpha}_i \cap \varsigma^{\beta}_j \neq 0, \\
0 & otherwise
\end{matrix}\right. .
\]
Here $Q_{\alpha \beta} =Q_{\beta \alpha}^T$ and $A = A^T \geq 0$.
It is well-known that the given finite volume method with two-point flux approximation provides a solution with second-order accuracy by space.

\subsection{Implicit approximation by time (coupled scheme)}

We apply an implicit scheme for time discretization to construct a discrete problem on the fine grid.  
Let $p^n_{\alpha,i} = p_{\alpha,i}(t_n)$ and $p^{n-1}_{\alpha,i} = p_{\alpha,i}(t_{n-1})$, where $t_n = n \tau$,  $n=1,2, ...$ and $\tau > 0$ be the fixed time step size.  
For approximation by time, we first apply backward Euler's approximation for time derivative and obtain an implicit scheme
\begin{equation}
\label{ct-mc}
M \frac{ p^n - p^{n-1} }{\tau} + A p^n = F^n,
\quad A = D+ Q, \quad 
n = 1,2,...
\end{equation}
with the following initial conditions
\[
p^0 = p_0, 
\]
This system is coupled and the size of system on the  fine-grid is  $N_h = \sum_{\alpha} N^{\alpha}_h$.

\begin{theorem}
\label{t:t2}
The solution of the discrete problem \eqref{ct-mc} is unconditionally stable and satisfies the following estimate
\begin{equation}
\label{t2}
||p^n||_{A}^2 \leq ||p^{n-1}||_{A}^2 
+ \frac{\tau}{2} ||F^n||^2_{\left(M + \frac{\tau}{2} A \right)^{-1}}.
\end{equation}
\end{theorem}
\begin{proof}
The equation \eqref{ct-mc} can be written as follows
\[
\left(M + \frac{\tau}{2} A \right) \frac{ p^n - p^{n-1} }{\tau} + \frac{1}{2} A (p^n  + p^{n-1}) = F^n.
\]
After multiplication by $ \frac{ p^n - p^{n-1} }{\tau}$, we obtain
\[
\left( \left(M + \frac{\tau}{2} A \right) \frac{ p^n - p^{n-1} }{\tau} , \frac{ p^n - p^{n-1} }{\tau} \right) 
+ \frac{1}{2}  \left( A (p^n + p^{n-1}),   \frac{ p^n - p^{n-1} }{\tau} \right) = \left(F^n, \frac{ p^n - p^{n-1} }{\tau} \right). 
\]
Next, using  a Cauchy inequality 
for the right-hand side
\[
\left(F^n, \frac{ p^n - p^{n-1} }{\tau} \right) \leq 
\left\| \frac{ p^n - p^{n-1} }{\tau} \right\|_{ \left(M + \frac{\tau}{2} A \right) }^2 
+ \frac{1}{4} 
\left\| F^n \right\|_{ \left(M + \frac{\tau}{2} A \right)^{-1}}^2,
\]
we obtain the following estimate for $A = A^T \geq 0$
\[
\frac{1}{2 \tau} \left( A (p^n + p^{n-1}),   (p^n - p^{n-1})  \right)
=
\frac{1}{2 \tau} (A p^n, p^n) + 
\frac{1}{2 \tau} (A p^{n-1}, p^{n-1}) 
\leq 
\frac{1}{4} 
\left\| F^n \right\|_{ \left(M + \frac{\tau}{2} A \right)^{-1}}^2,
\]
or
\[
(A p^n, p^n) \leq (A p^{n-1} p^{n-1})  + \frac{\tau}{2} 
\left\| F^n \right\|_{ \left(M + \frac{\tau}{2} A \right)^{-1}}^2.
\]
This estimate ensures the stability of the implicit scheme with respect to the initial condition and the right-hand side. 
\end{proof}

Note that a priory estimate \eqref{t2} can be written as follows
\[
||p^n||_{A}^2 \leq ||p^0||_{A}^2 
+ \frac{\tau}{2} \sum_{k=0}^n ||F^k||^2_{\left(M + \frac{\tau}{2} A \right)^{-1}}
\]
which is the similar to the estimate \eqref{t1}.

Presented regular implicit approximation leads to the large coupled system of equations. To decouple the system of equations, we apply semi-implicit approximation by time and solve smaller problems for each continuum separately.

\subsection{Decoupled schemes}

A regular implicit time approximation considered above leads to the solution of the coupled system of equations that requires the construction of the large discrete matrix $A$ on fine grid $\mathcal{T}_h$ with size  $N_h = \sum_{\alpha} N^{\alpha}_h$.
Moreover,  the multicontinuum problem properties usually have high contrast, for example, $k_f >> k_m$ for the fractured reservoirs ($k_f$ is the fracture continuum permeability and $k_m$ is the porous matrix permeability). Therefore, a large number of iterations are required in the iterative method to solve the linear equation system on each time step.
In this work, instead of the solution of the large coupled system of linear equations on each time step, we use an additive representation of the matrix to construct an uncoupled scheme.  We decouple solutions for each continuum to avoid the discontinuous nature of the porosity and permeability in multicontinuum problems by applying an additive representation of the system operator $A$ 
\[
A = A_0 + A_1
\]
with $A_0 = \text{diag}(A_{11}, A_{22}, ..., A_{LL})$ and $A_1 = A - A_0$. 
We approximate the coupling term from the previous time layer and obtain the following semi-implicit or explicit-implicit scheme
\begin{equation}
\label{si-mc}
M \frac{ p^n - p^{n-1} }{\tau} + A_0 p^n + A_1 p^{n-1} = F^n, 
n = 1,2,...
\end{equation}
with initial condition $p^0 = p_0$.  

In considered multicontinuum problem  \eqref{app-mc}, we have the following representation of operators (\textit{D-scheme})
\[
A_0 = \begin{pmatrix}
D_1+\sum_{\beta \neq 1} Q_{1\beta} & 0 & ... & 0 \\
0 & D_2+ \sum_{\beta \neq 2} Q_{2\beta} & ... &  0 \\
 ... &  ... &  ...  &   ...   \\
0 & 0 &  ... & D_L+\sum_{\beta \neq L} Q_{L \beta}
\end{pmatrix},
\]\[
A_1 = A - A_0 = 
 \begin{pmatrix}
0 & -Q_{12} & ...  & -Q_{1L} \\
-Q_{21} &0 & ... & -Q_{2L} \\
 ... &  ... &  ...  &   ...   \\
-Q_{L1}   & -Q_{L2} & ... & 0
\end{pmatrix}.
\]
Such representation separates coupling terms between continua and leads to the independent calculations of the problems in each continuum
\[
M_{\alpha} \frac{ p^n_{\alpha} - p^{n-1}_{\alpha} }{\tau}  + 
\left( 
D_{\alpha} + \sum_{\beta \neq \alpha} Q_{\alpha\beta}
\right) p^n_{\alpha} =
 F^{n} + 
 \sum_{\beta \neq \alpha} Q_{\alpha\beta}
p^{n-1}_{\beta}, 
\quad \alpha = 1,2,...,L.
\]
This system is decoupled, and the size of the system on the fine-grid for each continuum is $N^{\alpha}_h$.
In the considered additive representation, we have the following properties for operators
\begin{equation}
\label{dc-op}
A = A_0 + A_1  \geq 0, \quad
A_0 - A_1 \geq 0.
\end{equation}

\begin{theorem}
\label{t:t3}
If $A_0 - A_1 \geq 0$, the solution of the discrete problem \eqref{si-mc} is unconditionally stable and satisfies the following estimate
\begin{equation}
\label{t3}
||p^n||_{A}^2 \leq ||p^{n-1}||_{A}^2 
+ \frac{\tau}{2} ||F^n||^2_{\left(M + \frac{\tau}{2} (A_0 - A_1) \right)^{-1}}.
\end{equation}
\end{theorem}
\begin{proof}
The equation \eqref{si-mc} can be written as follows
\[
\left(M + \frac{\tau}{2} (A_0 - A_1) \right) \frac{ p^n - p^{n-1} }{\tau} + \frac{1}{2} A (p^n  + p^{n-1}) = F^n.
\]
After multiplication by $ \frac{ p^n - p^{n-1} }{\tau}$, we obtain
\[
\left( \left(M + \frac{\tau}{2} A \right) \frac{ p^n - p^{n-1} }{\tau} , \frac{ p^n - p^{n-1} }{\tau} \right) 
+ \frac{1}{2}  \left( A (p^n + p^{n-1}),   \frac{ p^n - p^{n-1} }{\tau} \right) = \left(F^n, \frac{ p^n - p^{n-1} }{\tau} \right). 
\]
Next, using  a Cauchy inequality 
for the right-hand side
\[
\left(F^n, \frac{ p^n - p^{n-1} }{\tau} \right) \leq 
\left\| \frac{ p^n - p^{n-1} }{\tau} \right\|_{ \left(M + \frac{\tau}{2}  (A_0 - A_1) \right) }^2 
+ \frac{1}{4} 
\left\| F^n \right\|_{ \left(M + \frac{\tau}{2}  (A_0 - A_1) \right)^{-1}}^2,
\]
we obtain the following estimate
\[
\frac{1}{2 \tau} \left( A (p^n + p^{n-1}),   (p^n - p^{n-1})  \right)
=
\frac{1}{2 \tau} (A p^n, p^n) + 
\frac{1}{2 \tau} (A p^{n-1} p^{n-1}) 
\leq 
\frac{1}{4} 
\left\| F^n \right\|_{ \left(M + \frac{\tau}{2}  (A_0 - A_1) \right)^{-1}}^2,
\]
or
\[
(A p^n, p^n) \leq (A p^{n-1} p^{n-1})  + \frac{\tau}{2} 
\left\| F^n \right\|_{ \left(M + \frac{\tau}{2}  (A_0 - A_1) \right)^{-1}}^2.
\]
This ensures the stability of the semi-implicit scheme with respect to the initial condition and the right-hand side. 
\end{proof}

Let us sort the continuum based on their permeability,  $k_1 < k_2 <...<k_L$, and use the calculated continuum solution in the solution of the following equation. Therefore, we can construct the following schemes
\begin{itemize}
\item (\textit{L-scheme})
\[
A_0 = \begin{pmatrix}
D_1+\sum_{\beta \neq 1} Q_{1\beta} & 0 & ... & 0 \\
-Q_{21} & D_2+ \sum_{\beta \neq 2} Q_{2\beta} & ... &  0 \\
 ... &  ... &  ...  &   ...   \\
-Q_{L1} & -Q_{L2} &  ... & D_L+\sum_{\beta \neq L} Q_{L \beta}
\end{pmatrix},
\]\[
A_1 = A - A_0 = 
 \begin{pmatrix}
0 & -Q_{12} & ...  & -Q_{1L} \\
0 &0 & ... & -Q_{2L} \\
 ... &  ... &  ...  &   ...   \\
0   & 0 & ... & 0
\end{pmatrix}.
\]
Here we first calculate a continuum with smaller permeability, and  $A_0$ is the lower-triangular matrix. 

\item (\textit{U-scheme})
\[
A_0 = \begin{pmatrix}
D_1+\sum_{\beta \neq 1} Q_{1\beta} & -Q_{12}  & ... & -Q_{1L} \\
0 & D_2+ \sum_{\beta \neq 2} Q_{2\beta} & ... &  -Q_{2L} \\
 ... &  ... &  ...  &   ...   \\
0 & 0 &  ... & D_L+\sum_{\beta \neq L} Q_{L \beta}
\end{pmatrix},
\]\[
A_1 = A - A_0 = 
 \begin{pmatrix}
0 & 0& ...  & 0 \\
-Q_{21} &0 & ... & 0 \\
 ... &  ... &  ...  &   ...   \\
-Q_{L1}   & -Q_{L2} & ... & 0
\end{pmatrix}.
\]
Here we calculate the first continuum with larger permeability, and  $A_0$ is the upper-triangular matrix. 
\end{itemize}

We note that, \textit{D-,L-} and \textit{U-schemes} are all satisfy condition \eqref{dc-op} and therefore unconditionally stable and estimate  \eqref{t3} is valid.

\section{Decoupling schemes for the coarse grid nonlocal multicontinuum approximation}


In this section, we extend the presented decoupling technique for multiscale approximation using a nonlocal multicontinuum (NLMC) method. The NLMC method is the accurate multiscale multicontinuum approximation technique based on the nonlocal representation of the diffusion part of the operator. We show that the system obtained using the NLMC method has the same size as a regular finite volume approximation on a coarse grid but provides a very accurate solution with a significant reduction of the discrete system size. First, we present a regular implicit time approximation for a multiscale multicontinuum system in the coarse grid. Then we extend a decoupling technique for multiscale multicontinuum approximation for flow problems in fractured media.

\subsection{Multiscale approximation by space using NLMC}

For accurate approximation by space on the coarse grid of the coupled system of equations,  we use a  nonlocal multicontinuum (NLMC) approach. We construct multiscale basis functions by solving the local problems in the local domain, satisfying the coupled flow equations subject to continuum separation constraints. Given constraints provide meaning to the coarse scale solution: the local solution has zero mean in another continuum except for the one for which it is formulated.
The resulting multiscale basis functions have spatial decay properties in local domains and can separate the continuum.
The resulting basis functions will be used to construct the upscaled model. The resulting approximation provides an accurate approximation by a nonlocal approximation of the fluxes.

Let $K^+_i$ be an oversampled region for the coarse cell $K_i$ obtained by enlarging $K_i$ by several coarse cell layers.
We  construct a set of basis functions $\psi^{i,\beta} = (\psi^{i,\beta} _1, \psi^{i,\beta} _2,...,\psi^{i,\beta} _L)$  in local domain $K_i^+$ ($\psi^{i,\beta}_{\alpha} \in K^{\alpha}_{i+}$) related to the each continuum using the following constrains
\[
\frac{1}{|K^{\alpha}_j|}\int_{K^{\alpha}_j} \psi^{i,\beta}_{\alpha} dx = \delta_{ij}  \delta_{\alpha \beta},  
\quad \forall K^{\alpha}_j \in K_i^+,
\]
where $K^{\alpha}_j = \Omega_{\alpha} \cap K_j$. 
The resulting function has a mean value of one on the coarse cell $K_i$ for the current continuum and has a mean value of zero on all other coarse cells within $K_i^+$ and all coarse cells for another continuum.  
For the construction of the multiscale basis functions $\psi^{i,\beta}$,  we solve the following constrained energy minimizing problem in the oversampled local domain ($K_i^+$) using a fine-grid approximation for the coupled system
\begin{equation}
\label{eq:basis}
\begin{pmatrix}
D_1^{i+}+ \sum_{\beta \neq 1} Q^{i+}_{1\beta} & ... & -Q^{i+}_{1L} 
& C^T_1 & ... & 0 \\
... & ... & ... & ... & ... & ... \\
-Q^{i+}_{L1} & ...& D^{i+}_L + \sum_{\beta \neq L} Q^{i+}_{L\beta} & 0  &  ... &  C^T_L \\
C_1 & ...& 0 & 0  & ...& 0 \\
... & ... & ... & ... & ... & ... \\
0 & ...& C_L &  0 & ...  & 0\\
\end{pmatrix}
\begin{pmatrix}
\psi^{i,\beta}_1 \\
... \\
\psi^{i,\beta}_L \\
\mu_1 \\
... \\
\mu_L \\
\end{pmatrix} =
\begin{pmatrix}
0 \\
...\\
0 \\
F^{\beta,i}_1 \\
...\\
F^{\beta,i}_L \\
\end{pmatrix}
\end{equation}
with  the zero Dirichlet boundary conditions on $\partial K^+_i$ for $\psi^{i,\beta}$ . 
Note that we used Lagrange multipliers $\mu_{\alpha}$ to impose the constraints. 
We set $F_{\alpha}^{\beta,i} = \{F^{\beta,i}_{\alpha,j}\}$,  where $F^{\beta,i}_{\alpha,j}$ is  related to the $K^{\beta}_j$ and  $F^{\beta,i}_{\alpha,j} =  \delta_{ij}  \delta_{\alpha \beta}$. 

By combining multiscale basis functions, we obtain the following multiscale space and projection matrix
\[
V_H = \text{span} \{ \psi^{i,\beta}=(\psi^{i,\beta}_1, \psi^{i,\beta}_2,...,\psi^{i,\beta}_L), \quad
\beta = \overline{1,L}, \quad
 i = \overline{1,N_c} \}, 
\]\[
R = \begin{pmatrix}
R_{11} & R_{12}& ...& R_{1L} \\
R_{21} & R_{22}&  ...& R_{2L} \\
 ...&  ...&  ...&  ... \\
R_{L1} & R_{L2}  &  ...& R_{LL}
\end{pmatrix},
\quad 
R_{\alpha \beta} = \begin{pmatrix}
\psi^{0,\alpha}_{\beta} \\
\psi^{1,\alpha}_{\beta} \\
... \\
\psi^{N_c,\alpha}_{\beta} 
\end{pmatrix},
\]
where $\psi^{i,\alpha}_{\beta} \in K^{\alpha}_{i+}$. 

To construct a coarse scale approximation, we use a projection approach and obtain the following approximation for matrix $A$ on the coarse grid
\[
\bar{A} = R A R^T.
\]

We remark that $\bar{p}_{\alpha}$ is the average cell solution on coarse grid element for continuum $\alpha$ and 
we have the following coarse scale coupled system for $\bar{p} = (\bar{p}_1, \bar{p}_2, ...,\bar{p}_L)$
\begin{equation}
\label{app-nlmc}
\bar{M} \frac{\partial \bar{p}}{\partial t} + \bar{A} \bar{p} = \bar{F}, \quad
\bar{A} = \bar{D} + \bar{Q},
\end{equation}
with 
\[
\bar{M} = \begin{pmatrix}
\bar{M}_1 & 0 & ... & 0 \\
0 & \bar{M}_2 & ... &  0 \\
 ... &  ... &  ...  &   ...   \\
0 & 0 &  ... & \bar{M}_L
\end{pmatrix},  \quad
\bar{Q} = \begin{pmatrix}
\sum_{\beta \neq 1} \bar{Q}_{1\beta} & -\bar{Q}_{12} & ...  & -\bar{Q}_{1L} \\
-\bar{Q}_{21} & \sum_{\beta \neq 2} \bar{Q}_{2\beta} & ... & -\bar{Q}_{2L} \\
 ... &  ... &  ...  &   ... &  \\
-\bar{Q}_{L1}   & -\bar{Q}_{L2} & ... & \sum_{\beta \neq L} \bar{Q}_{L \beta}
\end{pmatrix},\quad
\bar{F} = \begin{pmatrix}
\bar{F}_1 \\
\bar{F}_2 \\
 ...   \\
\bar{F}_L
\end{pmatrix}, 
\]
and 
\[
\bar{D} = RDR^T = \begin{pmatrix}
\bar{D}_{11} & \bar{D}_{12} & ... & \bar{D}_{1L} \\
\bar{D}_{21} & \bar{D}_{22} & ... &  \bar{D}_{2L} \\
 ... &  ... &  ...  &   ...   \\
\bar{D}_{L1} & \bar{D}_{L2} &  ... & \bar{D}_{LL}
\end{pmatrix}, 
\quad 
\bar{D}_{\alpha\beta} = \sum_{\zeta} R_{\alpha\zeta} D_{\zeta} R_{\beta\zeta}^T.
\] 
Due  to the constraints that we use for multiscale basis construction, we have 
$\bar{M}_{\alpha} = R_{\alpha\alpha} M_{\alpha} R_{\alpha\alpha}^T$, 
$\bar{Q}_{\alpha\beta} = R_{\alpha\alpha} Q_{\alpha\beta} R_{\beta\beta}^T$ and 
$\bar{F}_{\alpha} = R_{\alpha\alpha} F_{\alpha}$.

For $c_{\alpha} = \const$, $\sigma_{\alpha \beta} = \const$ and $f_{\alpha} = \const$ in each coarse cell $K_i^{\alpha}$,  the mass matrix, continuum coupling matrix and right-hand side vector can be directly calculated on the coarse grid and similar to the regular finite volume approximation
\[
\bar{M}_{\alpha} = \{{m}_{\alpha,ij}\}, \quad 
\bar{Q}_{\alpha \beta} = \{{q}_{\alpha \beta, ij}\}, \quad 
\bar{F}_{\alpha} = \{{f}_{\alpha,j} |K^{\alpha}_j|\}, \quad 
\]\[
{m}_{\alpha,ij} =
\left\{\begin{matrix}
{c}_{\alpha,i} |K^{\alpha}_i| & i = j, \\
0 & otherwise
\end{matrix}\right. ,  
\quad
q_{\alpha \beta,ij} =
\left\{\begin{matrix}
{\sigma}_{\alpha \beta,ij} &  K^{\alpha}_i \cap K^{\beta}_j \neq 0, \\
0 & otherwise
\end{matrix}\right. .
\]
However,  the matrix $D$ is non-local and provides a good approximation on the coarse grid due to the coupled multiscale basis construction. 
Similarly to the finite volume approximation, in the NLMC method, we have $M=M^T \geq 0$ and $A = A^T \geq 0$.

\subsection{Implicit approximation by time (coupled scheme)}

To approximate by time, we can use an implicit scheme  for multiscale approximation \eqref{app-nlmc}
\begin{equation}
\label{is-nlmc}
\bar{M} \frac{\bar{p}^{n} - \bar{p}^{n-1}}{\tau} + \bar{A} \bar{p}^{n} = \bar{F}^{n}, \quad n = 1,2,...
\end{equation}
with initial condition
\[
\bar{p}^0 = \bar{p}_0, 
\]
This system is coupled on the coarse grid and the size of system is  $N_H = \sum_{\alpha} N^{\alpha}_H$. 
Because the upscaled coarse grid problem is constructed based on the NLMC method,  the coarse grid matrices $\bar{M}$ and $\bar{Q}$ are the same as in finite volume approximation on the coarse grid.  Therefore, we can obtain similar estimates for the implicit scheme \eqref{is-nlmc}.
\begin{equation}
\label{t2c}
||\bar{p}^n||_{A}^2 \leq ||\bar{p}^{n-1}||_{\bar{A}}^2 
+ \frac{\tau}{2} ||\bar{F}^n||^2_{\left(\bar{M} + \frac{\tau}{2} \bar{A} \right)^{-1}}.
\end{equation}
The proof of the stability is similar to the fine grid system (see Theorem \ref{t:t2}).

\subsection{Decoupled schemes}

The system of equations constructed based on the implicit time approximation is coupled by nonlocal flux approximation and coupling term.  Similarly to the fine grid model, we use an additive operator representation to separate continuum and decouple calculations
\begin{equation}
\label{si-nlmc}
\bar{M} \frac{ \bar{p}^n - \bar{p}^{n-1} }{\tau} + \bar{A}_0 \bar{p}^n + \bar{A}_1 \bar{p}^{n-1} = \bar{F}^n, 
n = 1,2,...
\end{equation}
where 
\[
\bar{A} = \bar{A}_0 + \bar{A}_1.
\]
We use three choices of the operator $\bar{A}_0$ that separate calculations for each continuum
\begin{itemize}
\item (\textit{D-scheme})
\[
\bar{A}_0 =  \begin{pmatrix}
\bar{A}_{11}& 0 & ... & 0 \\
0 & \bar{A}_{22} & ... &  0 \\
 ... &  ... &  ...  &   ...   \\
0 & 0 &  ... & \bar{A}_{LL}
\end{pmatrix} = 
\begin{pmatrix}
\bar{D}_{11}+\sum_{\beta \neq 1} \bar{Q}_{1\beta} & 0 & ... & 0 \\
0 & \bar{D}_{22}+ \sum_{\beta \neq 2} \bar{Q}_{2\beta} & ... &  0 \\
 ... &  ... &  ...  &   ...   \\
0 & 0 &  ... & \bar{D}_{LL}+\sum_{\beta \neq L} \bar{Q}_{L \beta}
\end{pmatrix},
\]
where $\bar{A}_0$ is the diagonal matrix. 

\item (\textit{L-scheme})
\[
\bar{A}_0 = 
 \begin{pmatrix}
\bar{A}_{11}& 0 & ... & 0 \\
\bar{A}_{21} & \bar{A}_{22} & ... &  0 \\
 ... &  ... &  ...  &   ...   \\
\bar{A}_{L1} & \bar{A}_{L2} &  ... & \bar{A}_{LL}
\end{pmatrix} = 
\begin{pmatrix}
\bar{D}_{11}+\sum_{\beta \neq 1} \bar{Q}_{1\beta} & 0 & ... & 0 \\
\bar{D}_{21}-\bar{Q}_{21} & \bar{D}_{22}+ \sum_{\beta \neq 2} \bar{Q}_{2\beta} & ... &  0 \\
 ... &  ... &  ...  &   ...   \\
\bar{D}_{L1}-\bar{Q}_{L1}   & \bar{D}_{L2}-\bar{Q}_{L2}&  ... & \bar{D}_{LL}+\sum_{\beta \neq L} \bar{Q}_{L \beta}
\end{pmatrix},
\]
where we first calculate continuum with smaller permeability, and  $\bar{A}_0$ is the lower-triangular matrix. 

\item (\textit{U-scheme})
\[
\bar{A}_0 =
\begin{pmatrix}
\bar{A}_{11}& \bar{A}_{12} & ... & \bar{A}_{1L} \\
0 & \bar{A}_{22} & ... &  \bar{A}_{2L} \\
 ... &  ... &  ...  &   ...   \\
0 & 0 &  ... & \bar{A}_{LL}
\end{pmatrix} = 
\begin{pmatrix}
\bar{D}_{11}+\sum_{\beta \neq 1} \bar{Q}_{1\beta} & \bar{D}_{12}-\bar{Q}_{12}  & ... & \bar{D}_{1L}-\bar{Q}_{1L} \\
0 & \bar{D}_{22}+ \sum_{\beta \neq 2} \bar{Q}_{2\beta} & ... & \bar{D}_{2L} -\bar{Q}_{2L} \\
 ... &  ... &  ...  &   ...   \\
0 & 0 &  ... & \bar{D}_{LL}+\sum_{\beta \neq L} \bar{Q}_{L \beta}
\end{pmatrix},
\]
where we first calculate continuum with larger permeability, and  $\bar{A}_0$ is the upper-triangular matrix. 
\end{itemize}
Here $\bar{A}_1 = \bar{A} - \bar{A}_0$. 
The resulting system is decoupled, and the size of the coarse-grid system for each continuum is $N^{\alpha}_H$. 
All three schemes are unconditionally stable, and the following estimate is valid
\begin{equation}
\label{t3c}
||\bar{p}^n||_{\bar{A}}^2 \leq ||\bar{p}^{n-1}||_{\bar{A}}^2 
+ \frac{\tau}{2} ||\bar{F}^n||^2_{\left(\bar{M} + \frac{\tau}{2} (\bar{A}_0 - \bar{A}_1) \right)^{-1}}.
\end{equation}
The proof of the stability of the semi-implicit scheme with respect to the initial condition and the right-hand side is similar to the Theorem \ref{t:t3}.

\section{Numerical results}

We consider the model problem in multicontinuum media in the domain $\Omega = [0,1]^2$ with 25 fracture lines. The fracture distribution is depicted in Figure \ref{fig:mesh20}.  We set source term in fractures located in lower left area of the domain ($x \in [0.1, 0.15] \times [0.1, 0.15]$) and upper right area ($x \in [0.6, 0.65] \times [0.85, 0.9]$).  We set $f_f = q_w (p - p_w)$ with $p_w = 1.2$ and $q_w = 10^5$.  As initial condition, we set $p_{\alpha,0} = 1$ and  perform simulations for $T = 0.002$ with $N_T=50$ time steps, $\tau = T/N_T$.

\begin{figure}[h!]
\centering
\includegraphics[width=0.5 \textwidth]{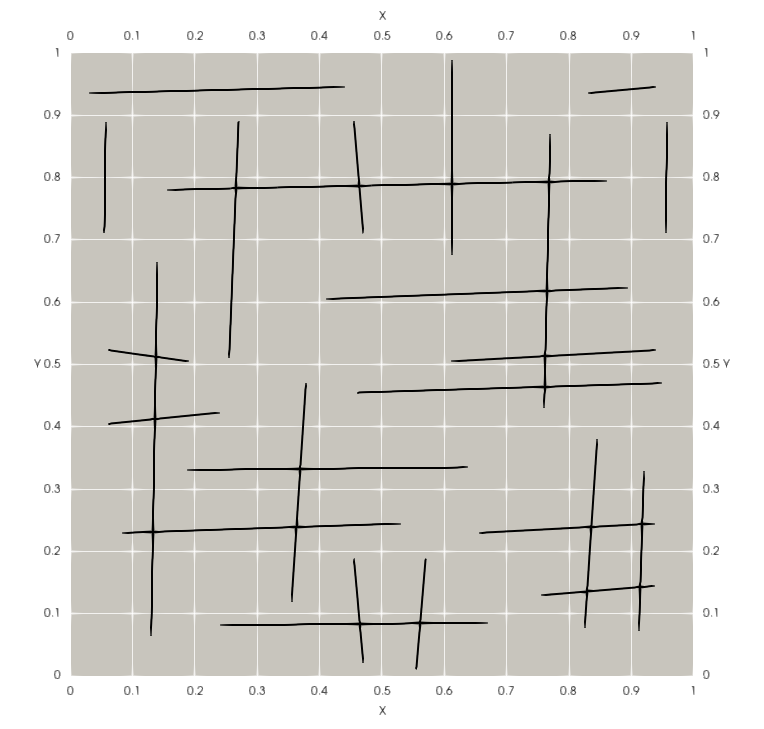}
\caption{Fractures geometry}
\label{fig:mesh20}
\end{figure}

We consider two test problems: 
\begin{itemize}
\item \textit{2C}: Two-continuum media, where we have a porous matrix and fracture continuum. We set $c_2 = 1$ and $k_2 = 10^6$ for fracture continuum,  $c_1 = 0.1$ and $k_1 = 1$  for porous matrix continuum.  
\item \textit{3C}: Three-continuum media, where we have porous matrix continuum, natural fracture continuum, and hydraulic fracture continuum. We set $c_3 = 1$ and $k_3 = 10^6$ for hydraulic fracture continuum,  $c_2 = 0.1$ and $k_2 = 1$ for natural fracture continuum,  and  $c_1 = 0.05$ and $k_1 = 10^{-3}$ for porous matrix continuum.  
\end{itemize}
The fine grid is $200 \times 200$ structured grid for domain $\Omega$ with quadratic cells. Fracture grid  (lower dimensional, 1D) contains  $2684$ cells.  For multiscale approach, we consider simulations on two coarse grids: (1) $20 \times 20$ coarse grid and (2) $40 \times 40$ coarse grid.   

Implementation is performed using python programming language, and  PETSc library \cite{balay2019petsc} for the solution of the linear system of equations at each time step. Because the resulting linear system is symmetric and positive definite, we use a typical iterative solver, a conjugate gradient (CG) iterative solver with ILU preconditioner. Simulations are performed on MacBook Pro (2.3 GHz Quad-Core Intel Core i7 with 32 GB 3733 MHz LPDDR4X). In this work, we did not compare different iterative solvers or preconditioners. We will consider it in future works for large three-dimensional problems. The research focuses on the decoupling schemes, where we perform numerical investigation for regular finite volume approximation and coarse grid approximation based on the nonlocal multicontinuum method (NLMC).

\subsection{Decoupling schemes for the fine grid finite volume approximation}

We first present numerical results for fine grid approximation using the coupled scheme.  
To compare coupled and decoupled methods, we take the solution of the coupled system as a reference solution and calculate the relative error  in percentage on the fine grid
\[
e_h^n = \frac{||p^n - \tilde{p}^n||_{L_2}}{||p^n||_{L_2}} \times 100 \%, \quad 
||p||_{L_2} =  \sqrt{(p, p)} 
\]
where $n$ is the time layer, $p$ is the reference solution (coupled scheme), and $\tilde{p}$ is the solution using the decoupled scheme.

\begin{figure}[h!]
\centering
\includegraphics[width=0.32 \textwidth]{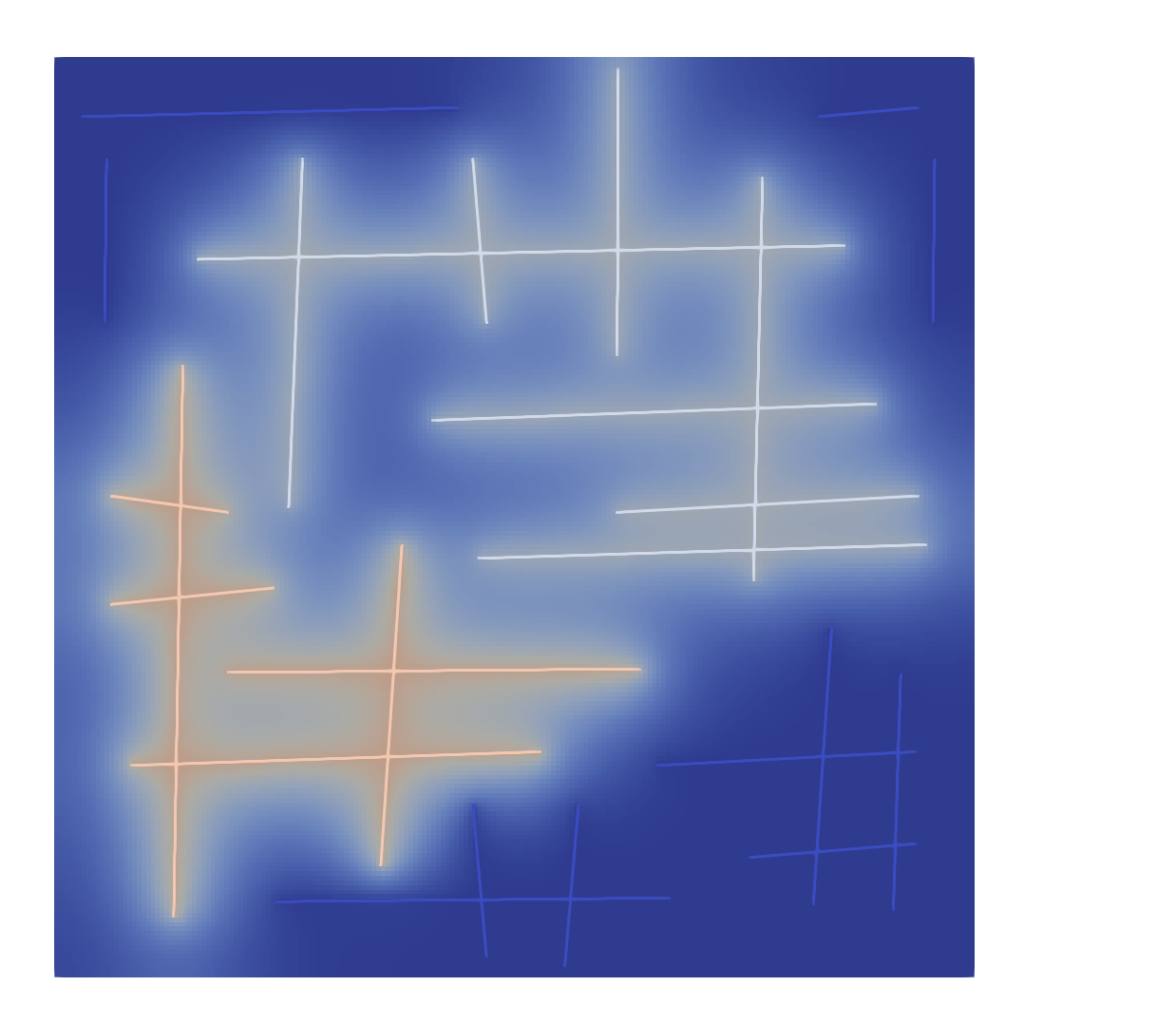}
\includegraphics[width=0.32 \textwidth]{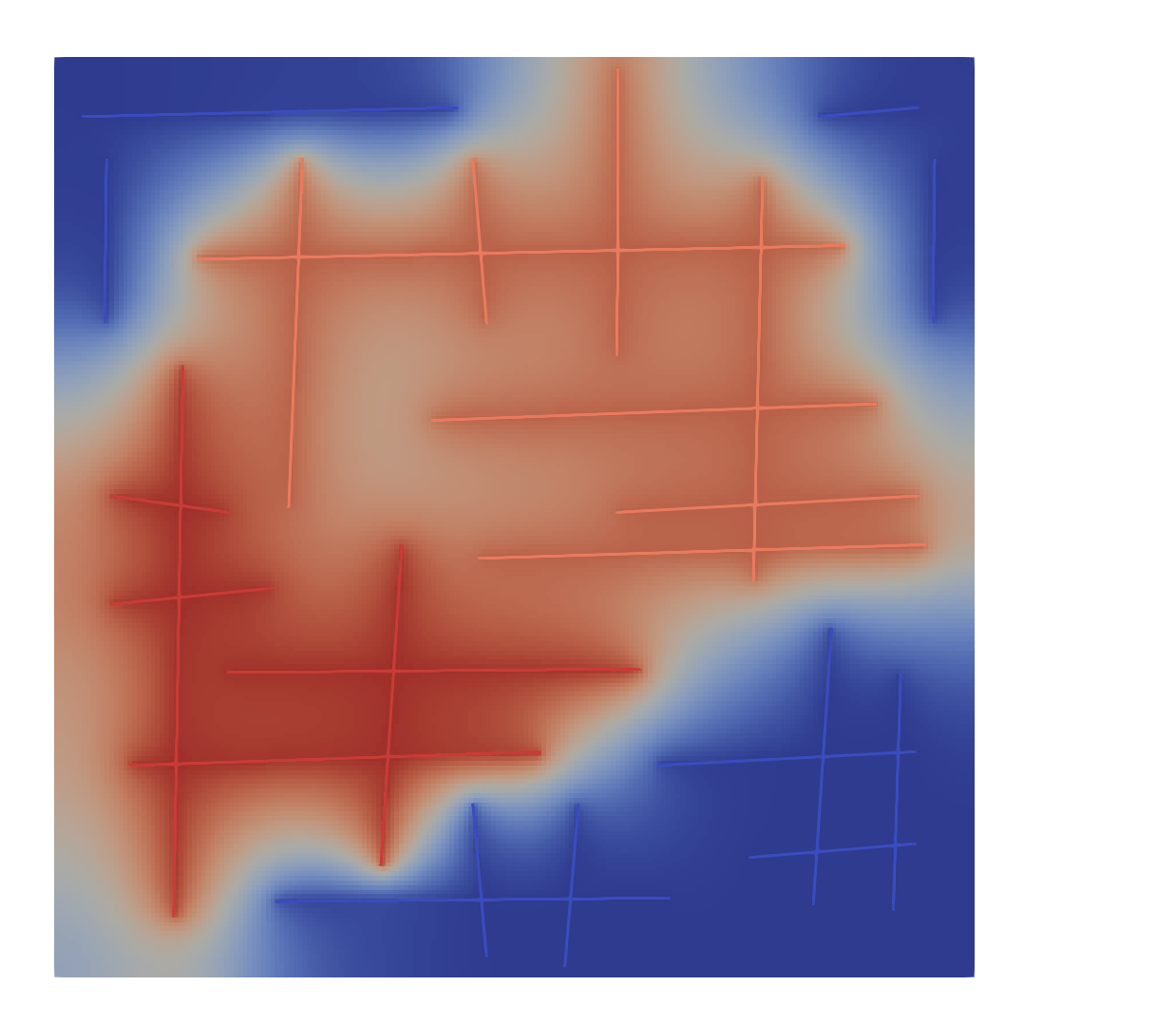}
\includegraphics[width=0.32 \textwidth]{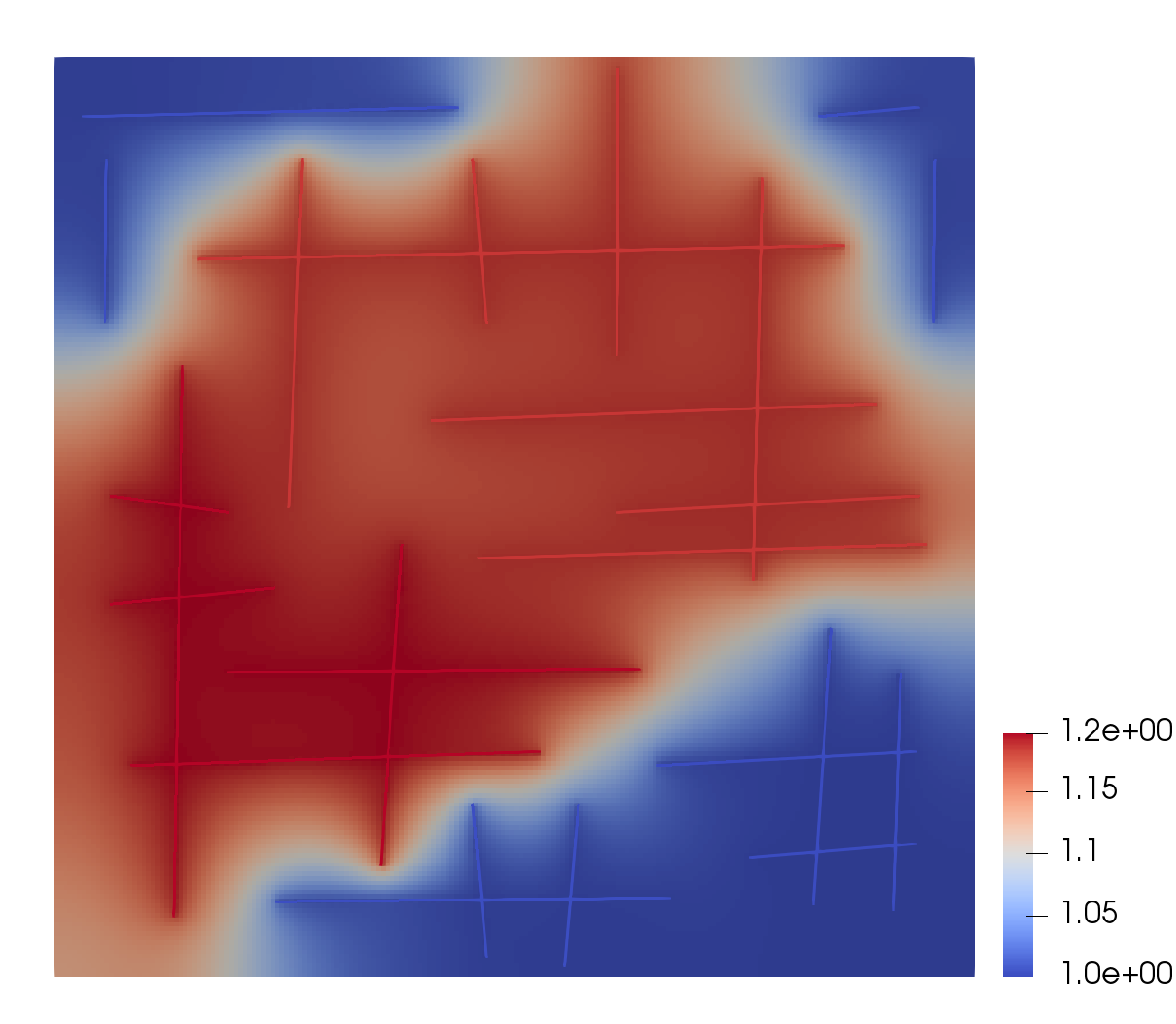}
\caption{Solution on the fine grid using the coupled scheme for two-continuum media (\textit{2C}). Solution $p^n$ for $n=10, 30$ and $50$ (from left to right)}
\label{fig:2c-u}
\end{figure}

\begin{figure}[h!]
\centering
\includegraphics[width=0.32 \textwidth]{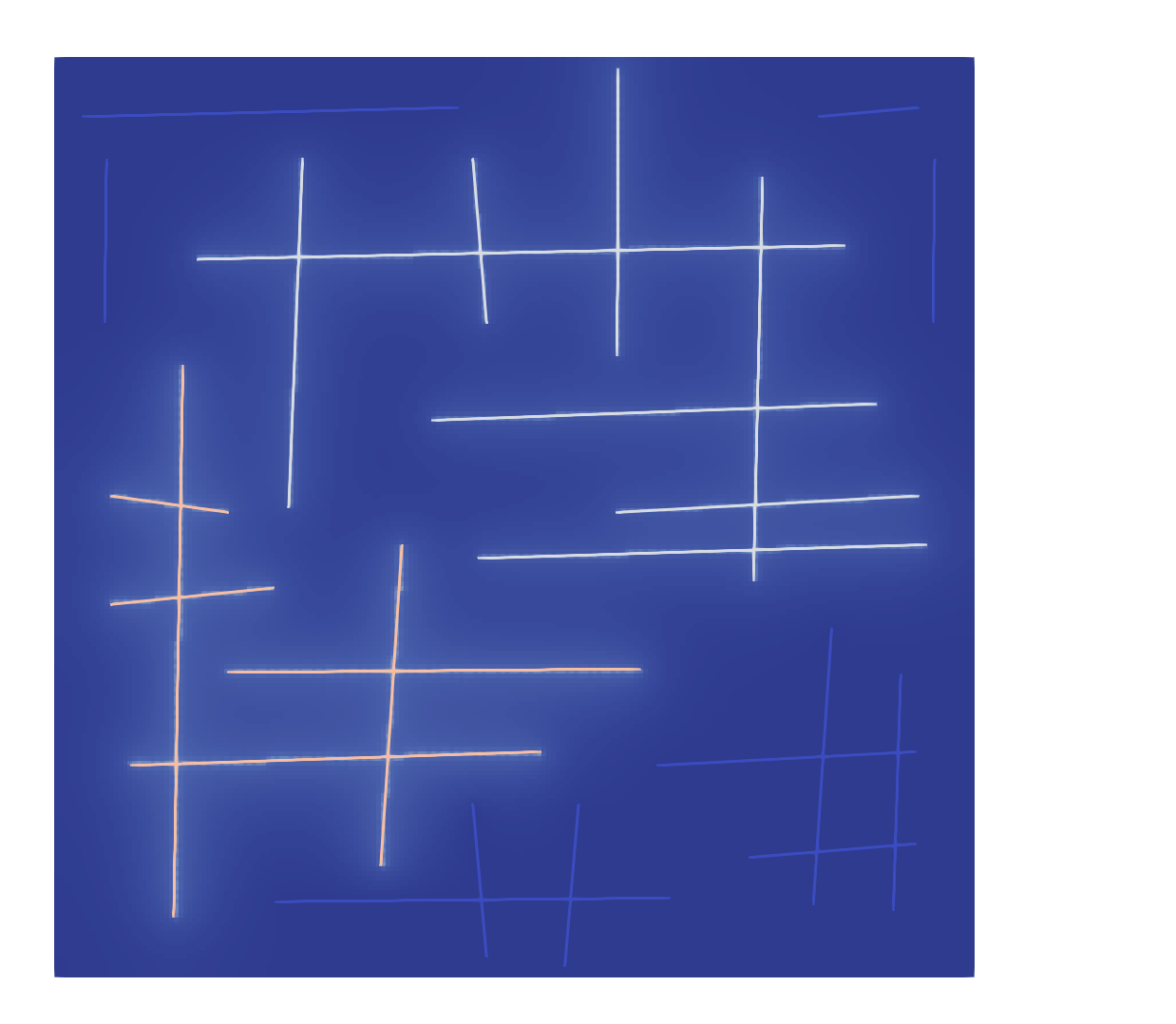}
\includegraphics[width=0.32 \textwidth]{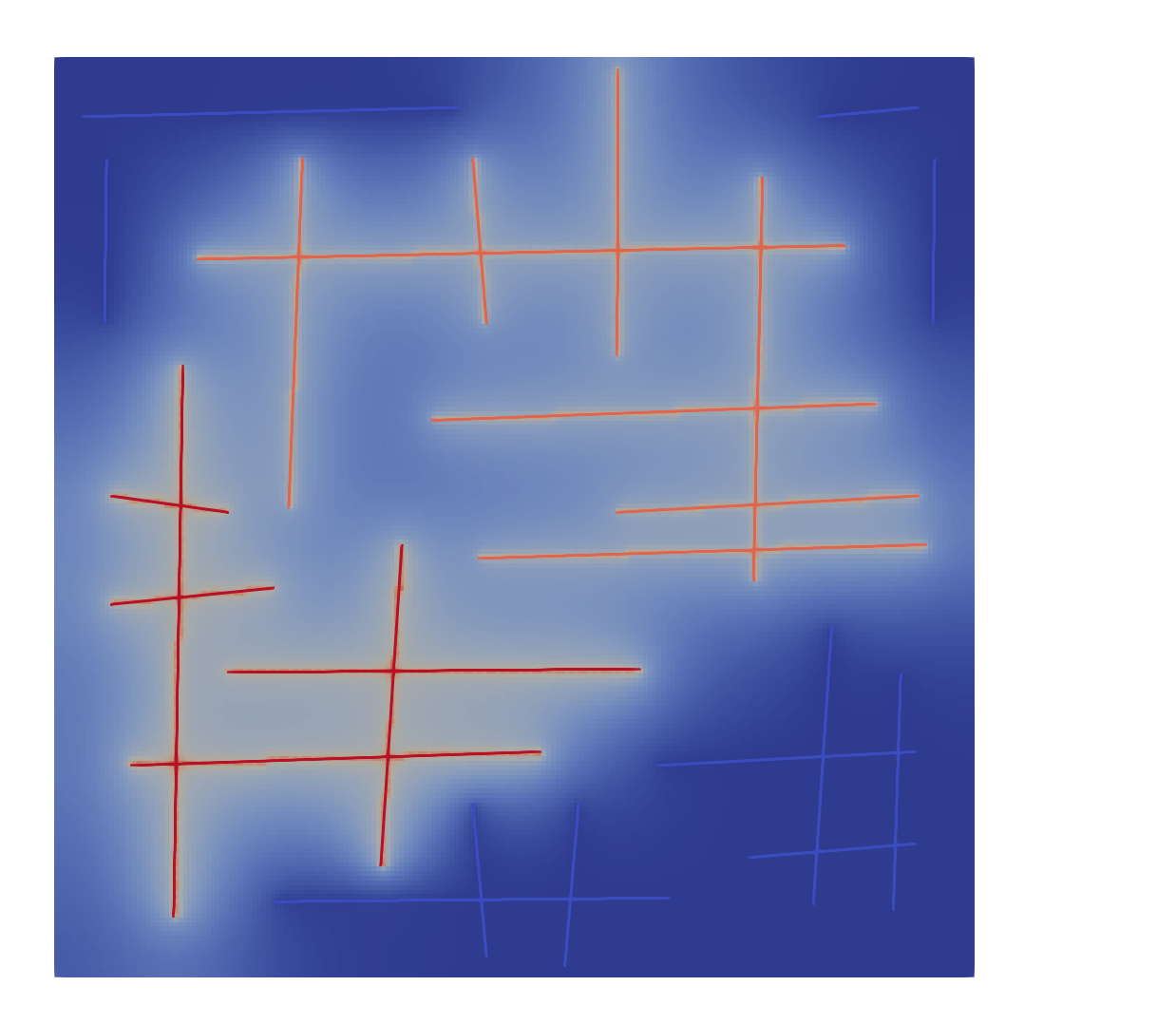}
\includegraphics[width=0.32 \textwidth]{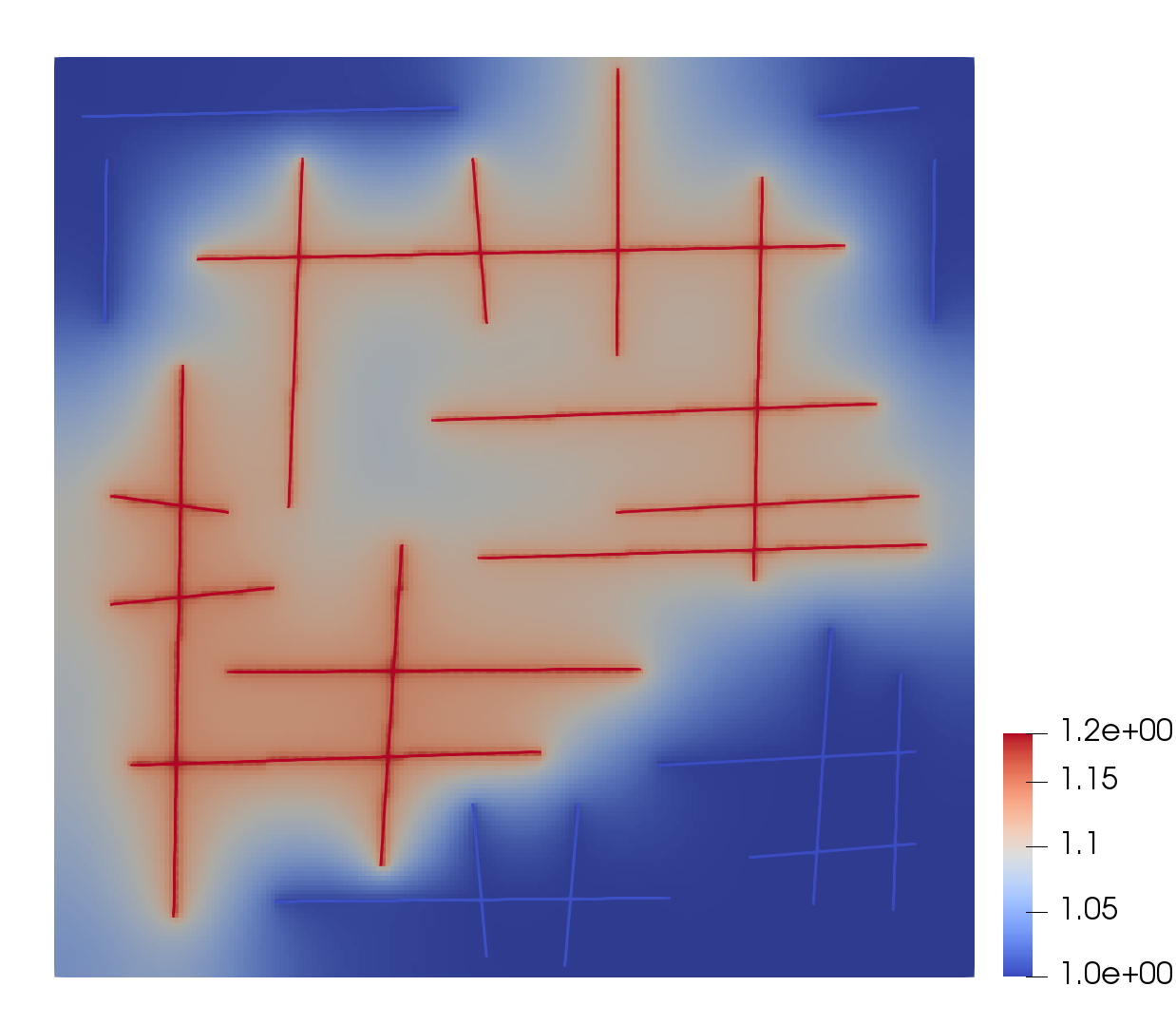}\\
\includegraphics[width=0.32 \textwidth]{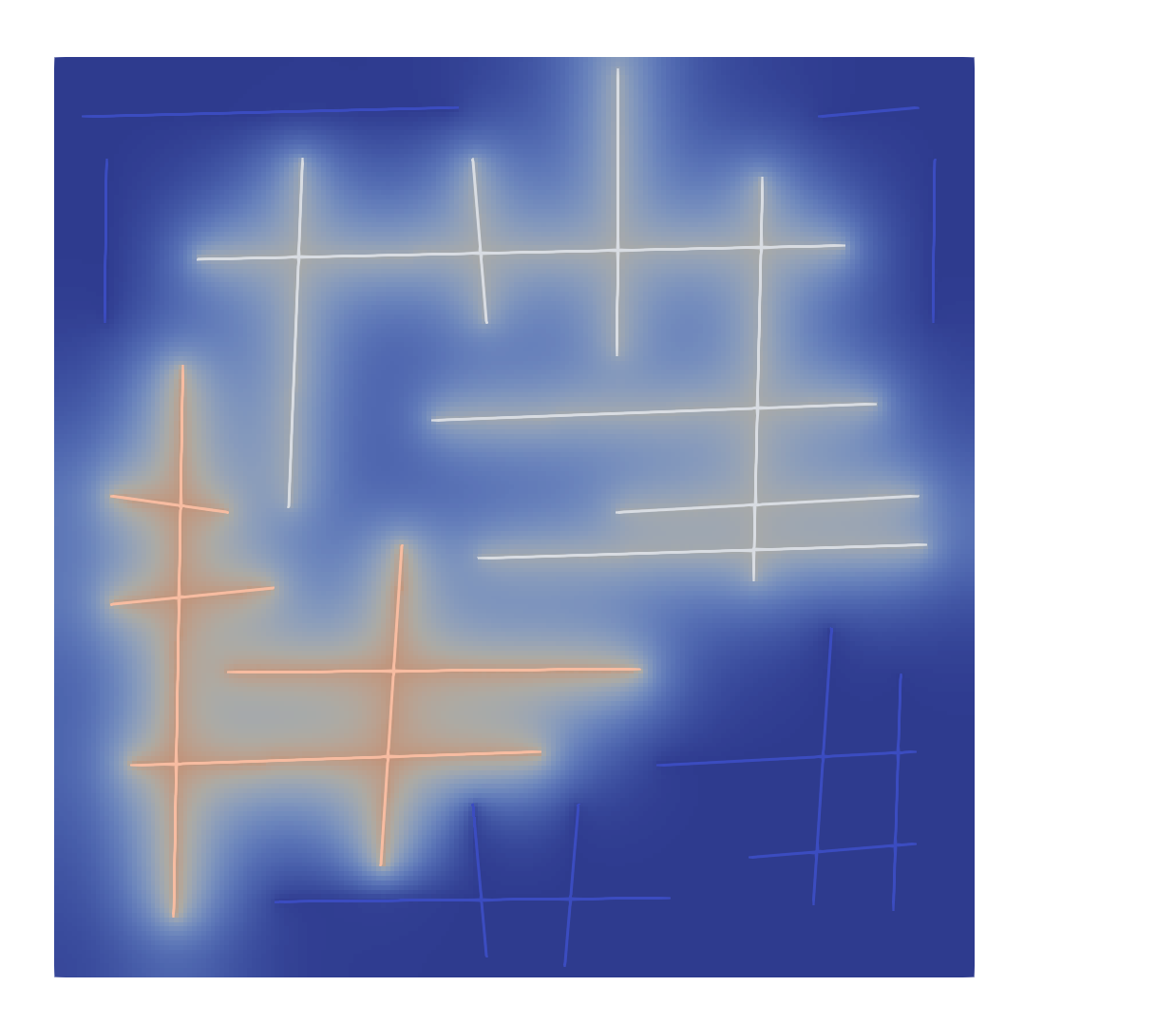}
\includegraphics[width=0.32 \textwidth]{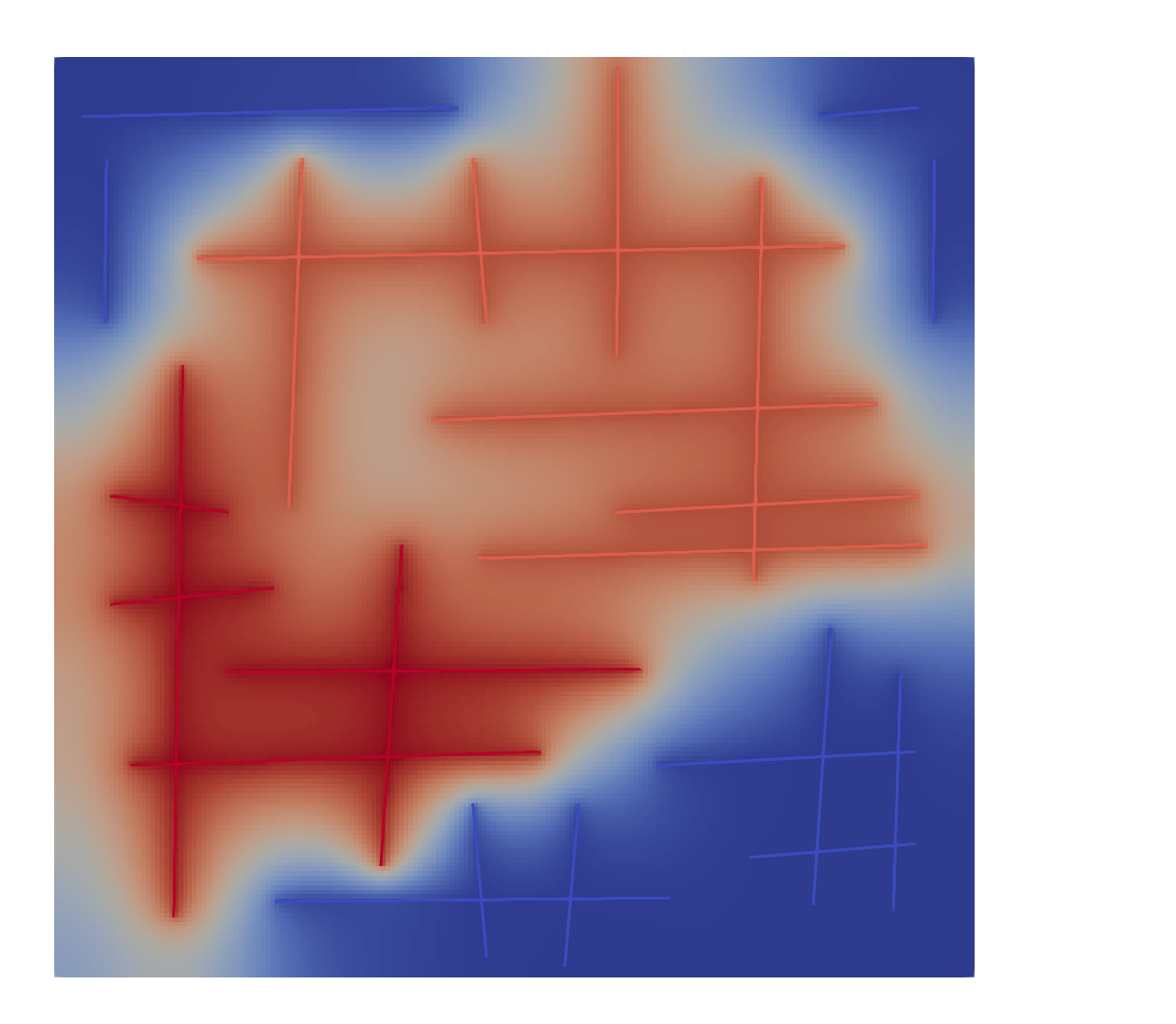}
\includegraphics[width=0.32 \textwidth]{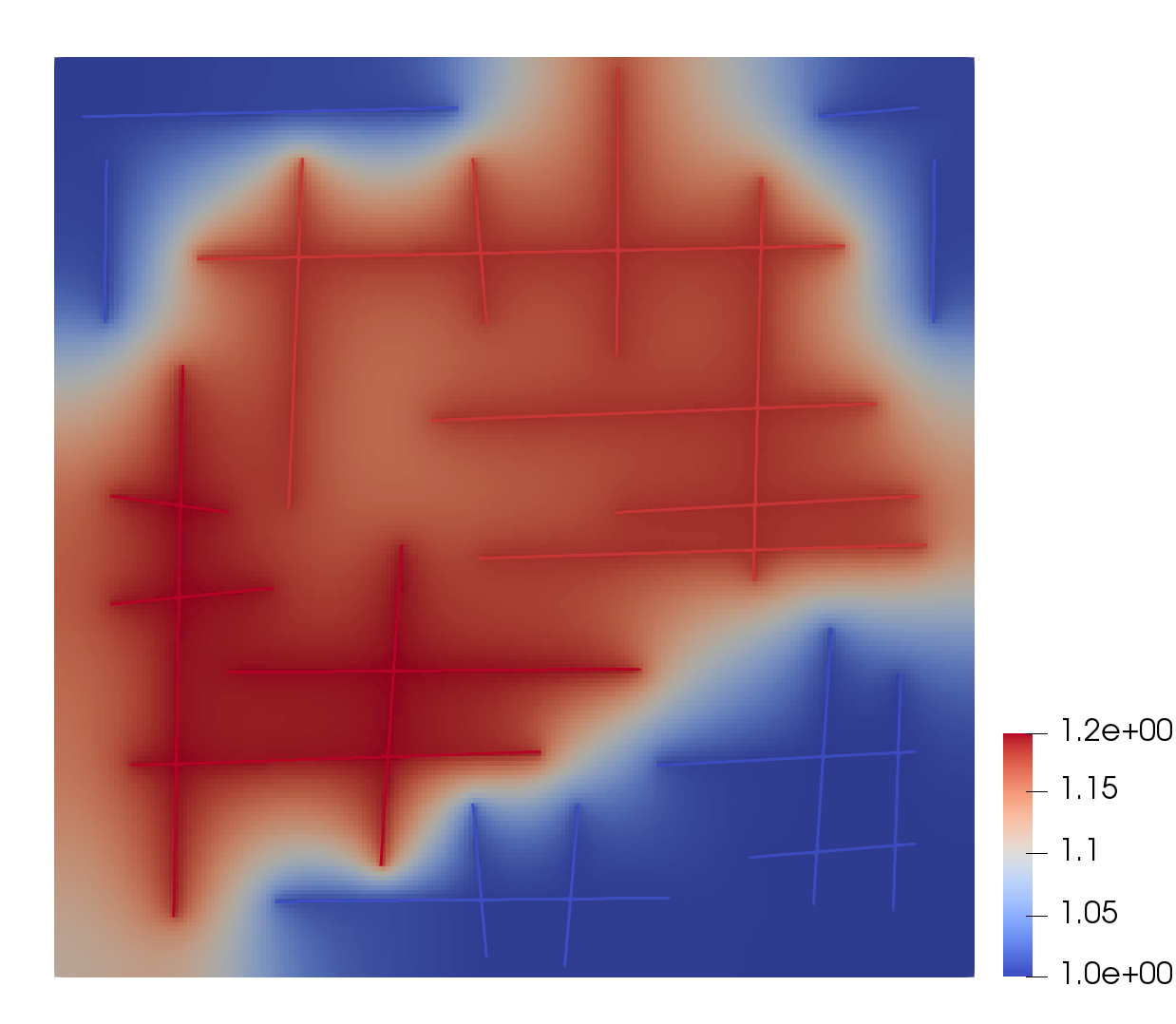}
\caption{Solution on the fine grid using the coupled scheme for three-continuum media (\textit{3C}). Solution $p^n$ for $n=10, 30$ and $50$ (from left to right). 
First row: porous matrix continuum. 
Second row: natural fracture continuum}
\label{fig:3c-u}
\end{figure}

In Figure \ref{fig:2c-u} and \ref{fig:3c-u},  we present solution for two- and three--continuum media, respectively. The solution $p^n$ is shown at three time layers $n=10, 30$ and $50$.  Numerical simulations were performed using a coupled scheme. 
In coupled scheme, we solve large coupled system of equations that have $DOF_h = 42684$ for two-continuum problem (\textit{2C}) and $DOF_h = 82684$ for three-continuum problem (\textit{3C}).  
Solution time is 46.3 sec and 105.6 for \textit{2C} and \textit{3C}, respectively.  Average number of iterations for solution of the linear system of equations at each time layer is  $\bar{N}_{it} = 1146.46$ for \textit{2C} and $\bar{N}_{it} = 1118.5$ for \textit{3C} (see Tables {\ref{tab-2c} and {\ref{tab-2c}).   Note that the average number of iterations equals the total number of iterations divided by the number of time steps.

\begin{figure}[h!]
\centering
\includegraphics[width=0.49 \textwidth]{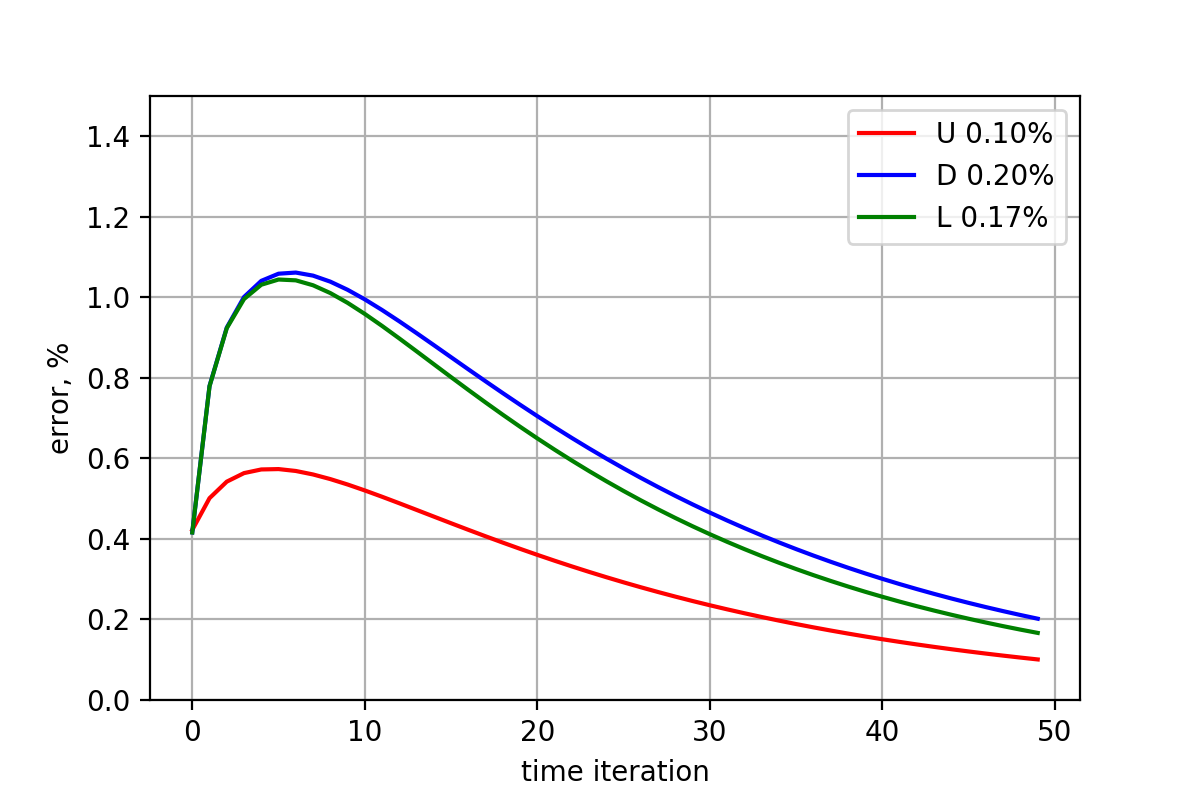}
\includegraphics[width=0.49 \textwidth]{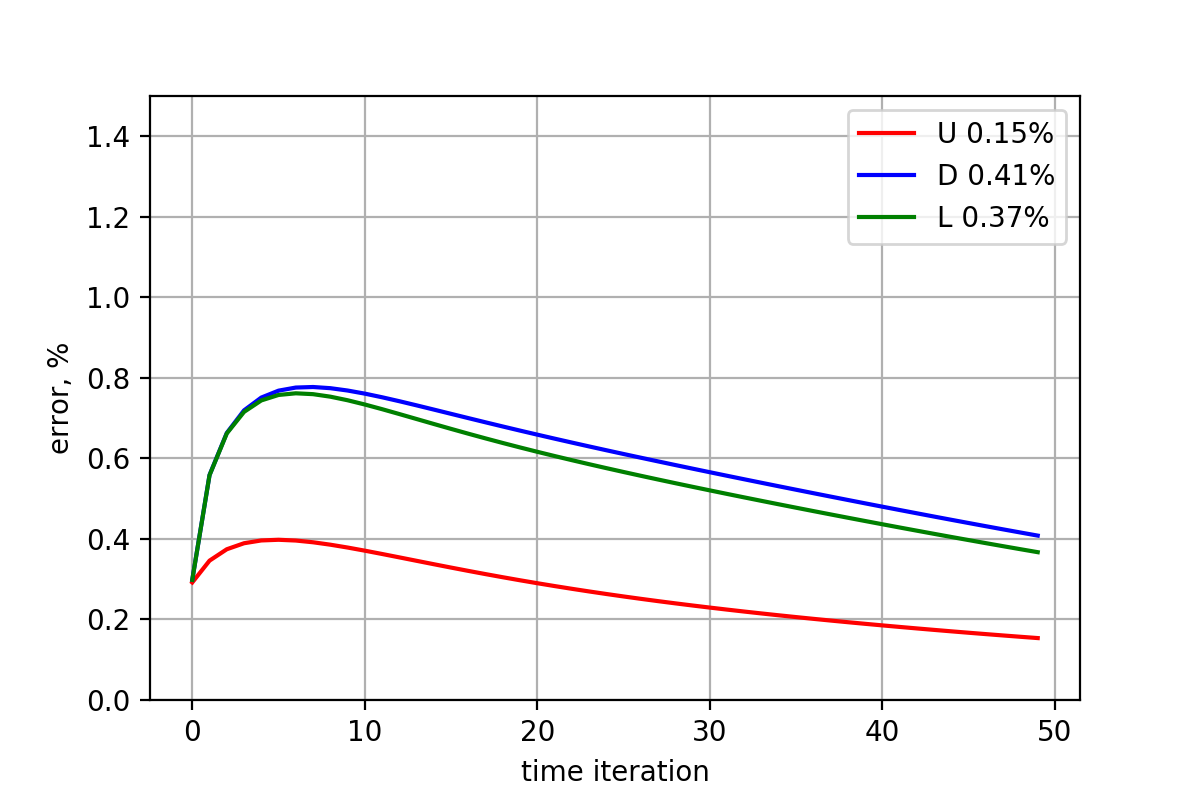}
\caption{Dynamic of the error (in percentage) for decoupled schemes.  The label is given with the error at the final time. 
Left: two-continuum media (\textit{2C}).
Right: three-continuum media (\textit{3C})}
\label{fig:dc}
\end{figure}

\begin{table}[h!]
\centering
\begin{tabular}{|c|c|c|cc|}
\hline
 \multicolumn{5}{|c|}{Two-continuum media (\textit{2C})} \\ 
\hline
& time$_{tot}$(sec) & $e_h$ (\%)
& time$_1$  ($\bar{N}_{it,1}$)  
& time$_2$  ($\bar{N}_{it,2}$)   \\
\hline
\textit{Coupled}	& 46.3 & - & \multicolumn{2}{|c|}{46.3 (1146.46)}  	\\
\textit{L-scheme} 	& 2.46 	& 0.17 \% 	& 1.35 (35.0) & 1.11 (892.2) 	\\
\textit{D-scheme} 	& 2.44 	& 0.20 \%	& 1.34 (35.0) & 1.10 (899.3)  \\
\textit{U-scheme} 	& 2.40 	& 0.10 \%	& 1.34 (35.0) & 1.06 (897.1) \\
\hline
\end{tabular}
\caption{Time of the solution and the average number of iterations for two-continuum media (\textit{2C}).  Coupled and decoupled schemes}
\label{tab-2c}
\end{table}

\begin{table}[h!]
\centering
\begin{tabular}{|c|c|c|ccc|}
\hline
 \multicolumn{6}{|c|}{Three-continuum media (\textit{3C})} \\ 
\hline
& time$_{tot}$(sec) & $e_h$ (\%)
& time$_1$  ($\bar{N}_{it,1}$)  
& time$_2$  ($\bar{N}_{it,2}$)  
& time$_3$  ($\bar{N}_{it,3}$)   \\
\hline
\textit{Coupled}	& 105.59 & - & \multicolumn{3}{|c|}{105.59 (1118.5)}  \\
\textit{L-scheme} 	& 2.52 & 0.37 \%	
& 0.15 (3.0) & 1.32 (35.0) & 1.05 (895.6) \\
\textit{D-scheme} 	& 2.58 	& 0.41 \% 
& 0.16 (3.0) & 1.36 (35.0) & 1.06 (890.3) \\
\textit{U-scheme} 	& 2.67 	& 0.15 \%
& 0.15 (3.0) & 1.39 (35.0) & 1.13 (897.1) \\
\hline
\end{tabular}
\caption{Time of the solution and the average number of iterations for three-continuum media (\textit{3C}).   Coupled and decoupled schemes}
\label{tab-3c}
\end{table}

In Figure \ref{fig:dc}, we present dynamic of the relative error for  three decoupling schemes: \textit{L}, \textit{D} and \textit{U}-schemes.  
Note that we sort continua in ascending order based on their permeability.  Therefore in  \textit{L-scheme}, we first solve a problem with lower permeability  (porous matrix).  For the  \textit{U-scheme}, we first solve a problem for the continuum related to the higher permeability or lower-dimensional fractures.   
In decoupled schemes,  we solve an equation for each continuum separately. 
For two-continuum problem,  the fist continuum is the porous matrix defined in domain $\Omega$ with $DOF_{h,1} = 40000$ and the second continuum is the lower-dimensional fractures with $DOF_{h,2} = 2684$ 
For the three-continuum problem, the fist continuum is the porous matrix  in $\Omega$ with $DOF_{h,1} = 40000$,  the second continuum  is the natural fractures in $\Omega$ with $DOF_{h,2} = 40000$, and the third continuum is the embedded fractures with $DOF_{h,3} = 2684$. 
From Figure \ref{fig:dc}, we observe that all schemes provide good results with small errors (less than 1\%).  However,  the  \textit{U-scheme}  give a better results.

Solution time with number of iterations are presented in Tables {\ref{tab-2c} and \ref{tab-3c}.   
We present the total time of solution (time$_{tot}$) for coupled and decoupled schemes with relative error in percentage at the final time.  
For decoupled scheme, we also present solution time related to each continuum equation with average number of iterations (time$_{\alpha}$  and $\bar{N}_{it,\alpha}$, $\alpha=1,2,3$). 
We note that the main part of the system is filled by continuum that defined in the  domain $\Omega$(for example,  $DOF_{h,1} = 40000$ in $\Omega$ and $DOF_{h,2} = 2684$ for lower-dimensional fracture domain $\gamma$ in \textit{2C} model).  
By system decoupling, we obtain a separate equation for each continuum. Therefore the number of iterations for the equation defined in $\Omega$ (less permeable domain than lower-dimensional fracture network) becomes smaller and reduces the calculation time. 
We have 2.5 sec of the solution time for all decoupled schemes, which is 19 times faster than the solution using the coupled scheme for \textit{2C} model and 39 times faster for \textit{3C} model. Moreover, the difference (error) between solutions is very small.

\subsection{Decoupling schemes for the coarse grid nonlocal multicontinuum approximation}

Next, we consider the solution of the problem on the coarse grid. We use a nonlocal multicontinuum (NLMC) method to construct a very accurate approximation on the coarse grid.  
We take a fine-grid solution with the coupled scheme as a reference solution. 
To compare coupled and decoupled methods for multiscale coarse grid approximation, we calculate relative error  in percentage on the coarse grid
\[
e_H^n = \frac{||\bar{p}^n - \tilde{p}^n||_{L_2}}{||\bar{p}^n||_{L_2}} \times 100 \%,
\]
where $n$ is the time layer, $\bar{p}$ is the reference solution (average on a coarse grid), and $\tilde{p}$ is the multiscale solution using the NLMC method for coupled and decoupled schemes.

\begin{figure}[h!]
\centering
$20 \times 20$ coarse grid\\
\includegraphics[width=0.32 \textwidth]{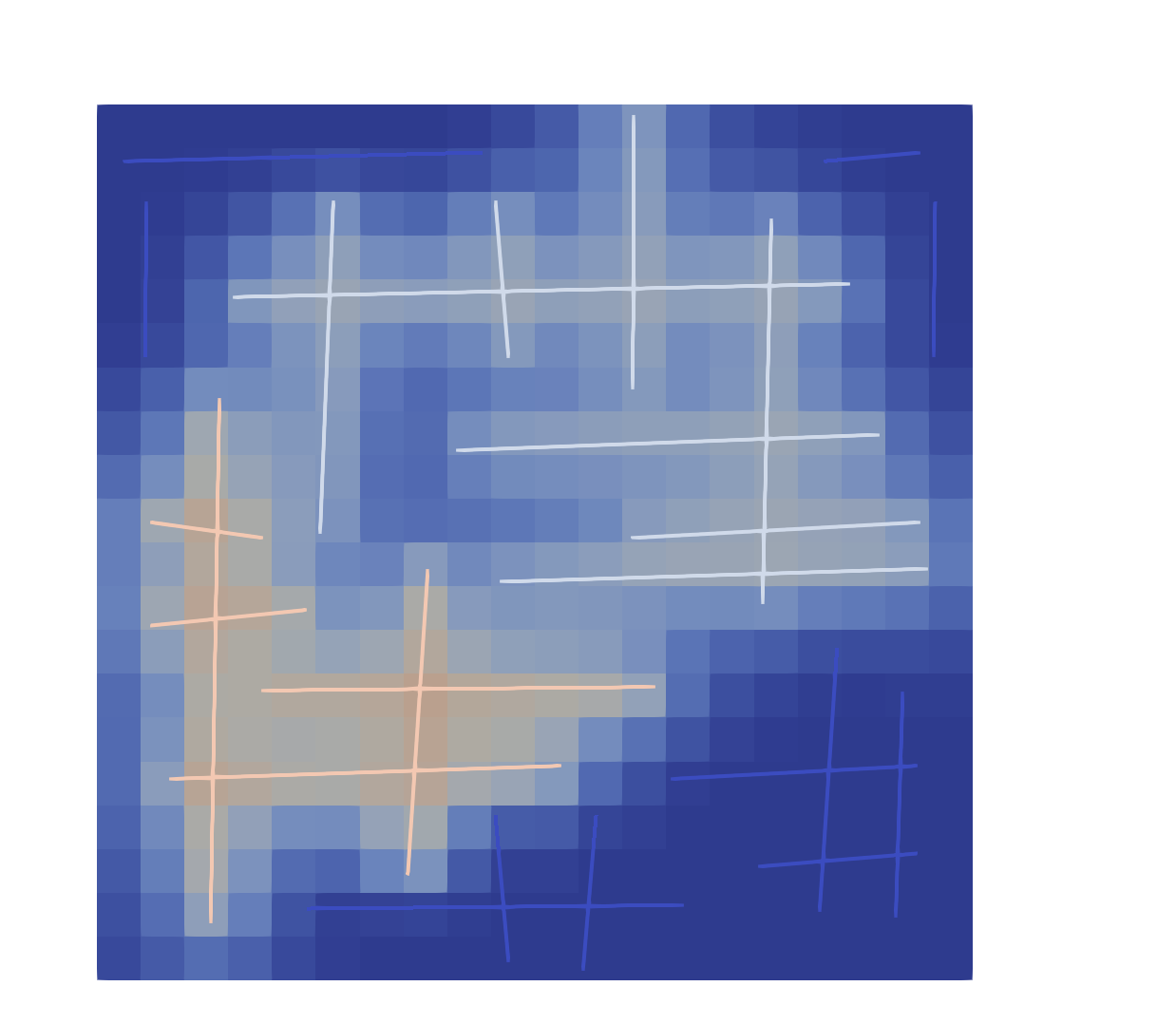}
\includegraphics[width=0.32 \textwidth]{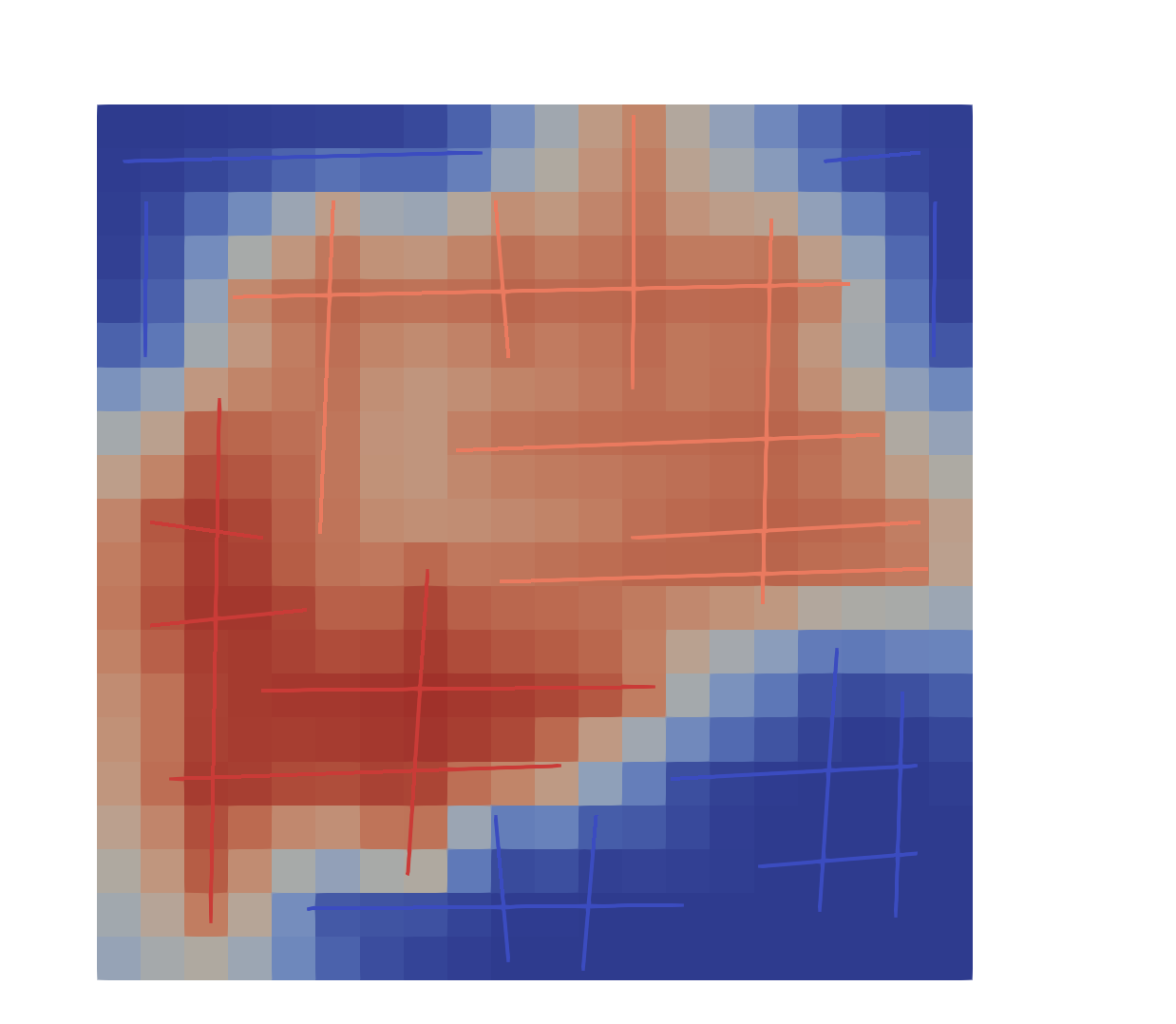}
\includegraphics[width=0.32 \textwidth]{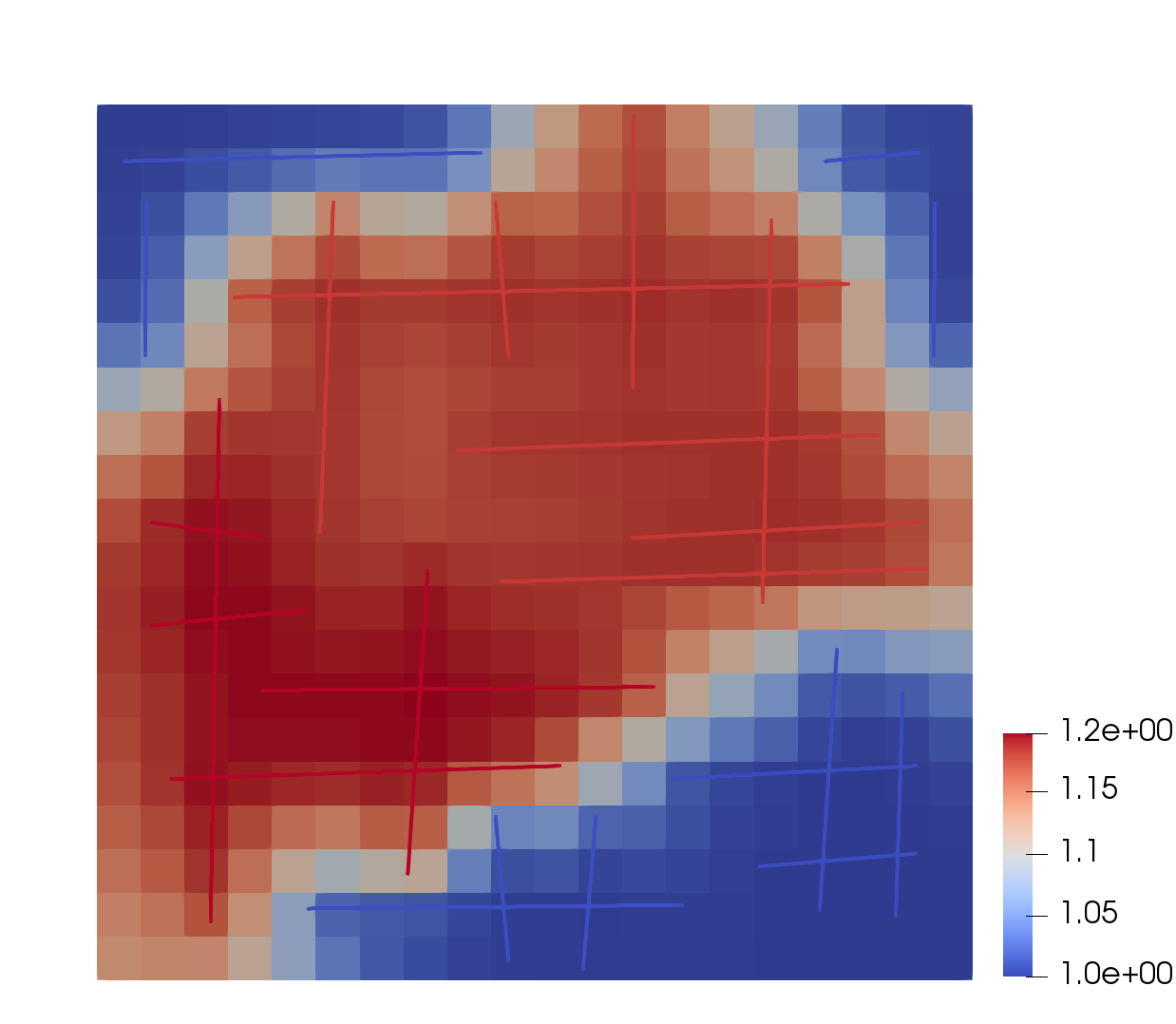}\\
$40 \times 40$ coarse grid\\
\includegraphics[width=0.32 \textwidth]{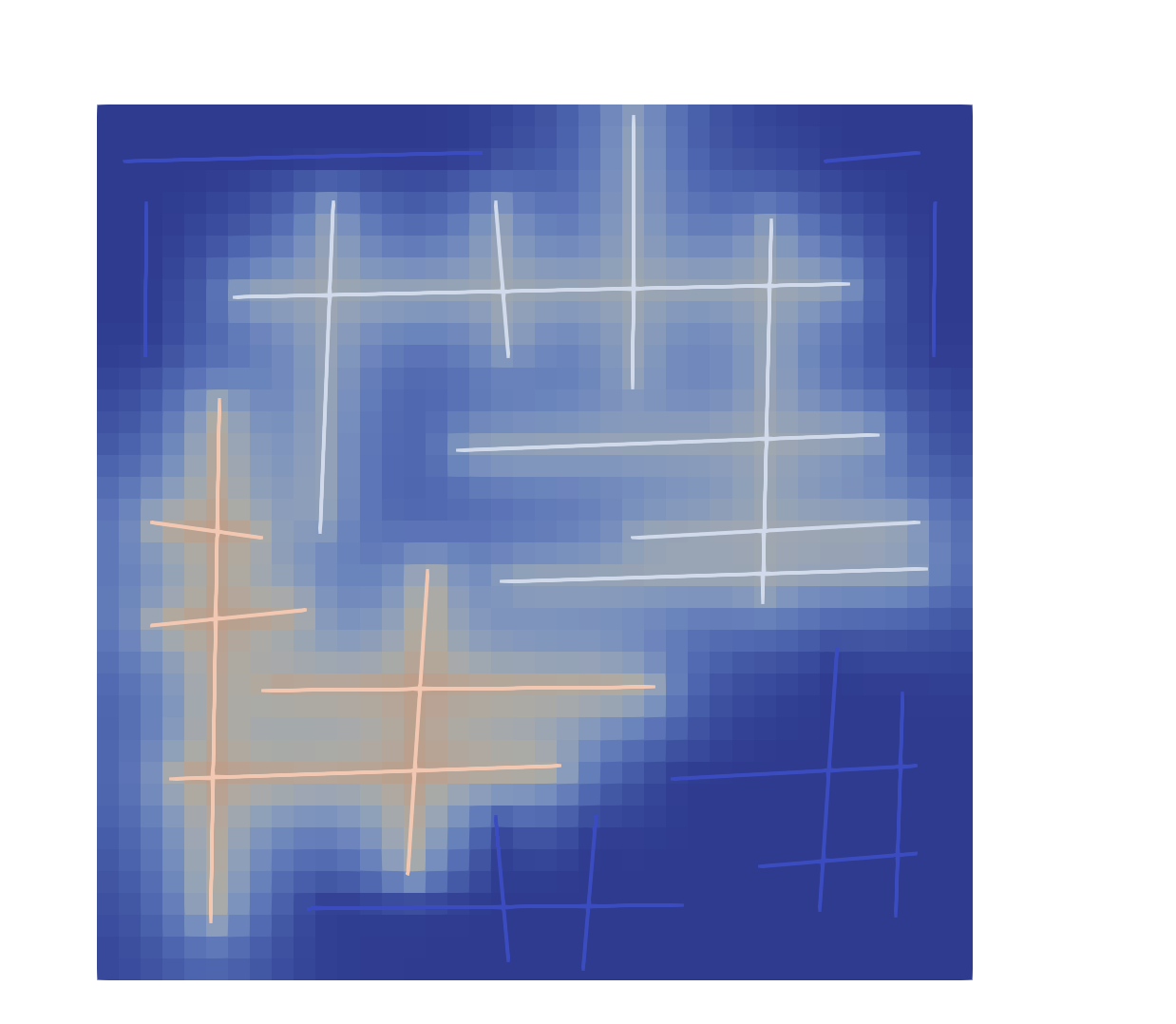}
\includegraphics[width=0.32 \textwidth]{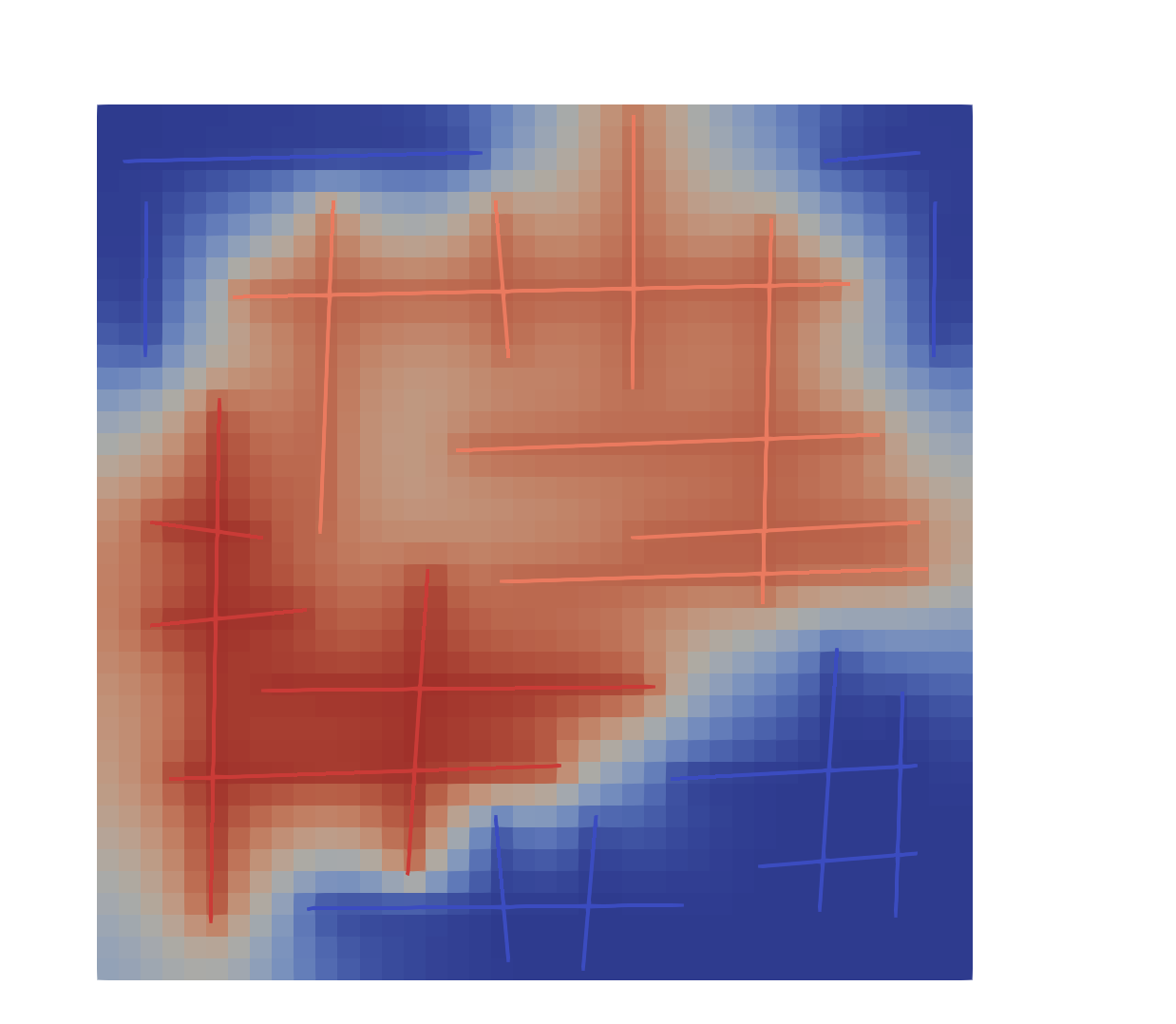}
\includegraphics[width=0.32 \textwidth]{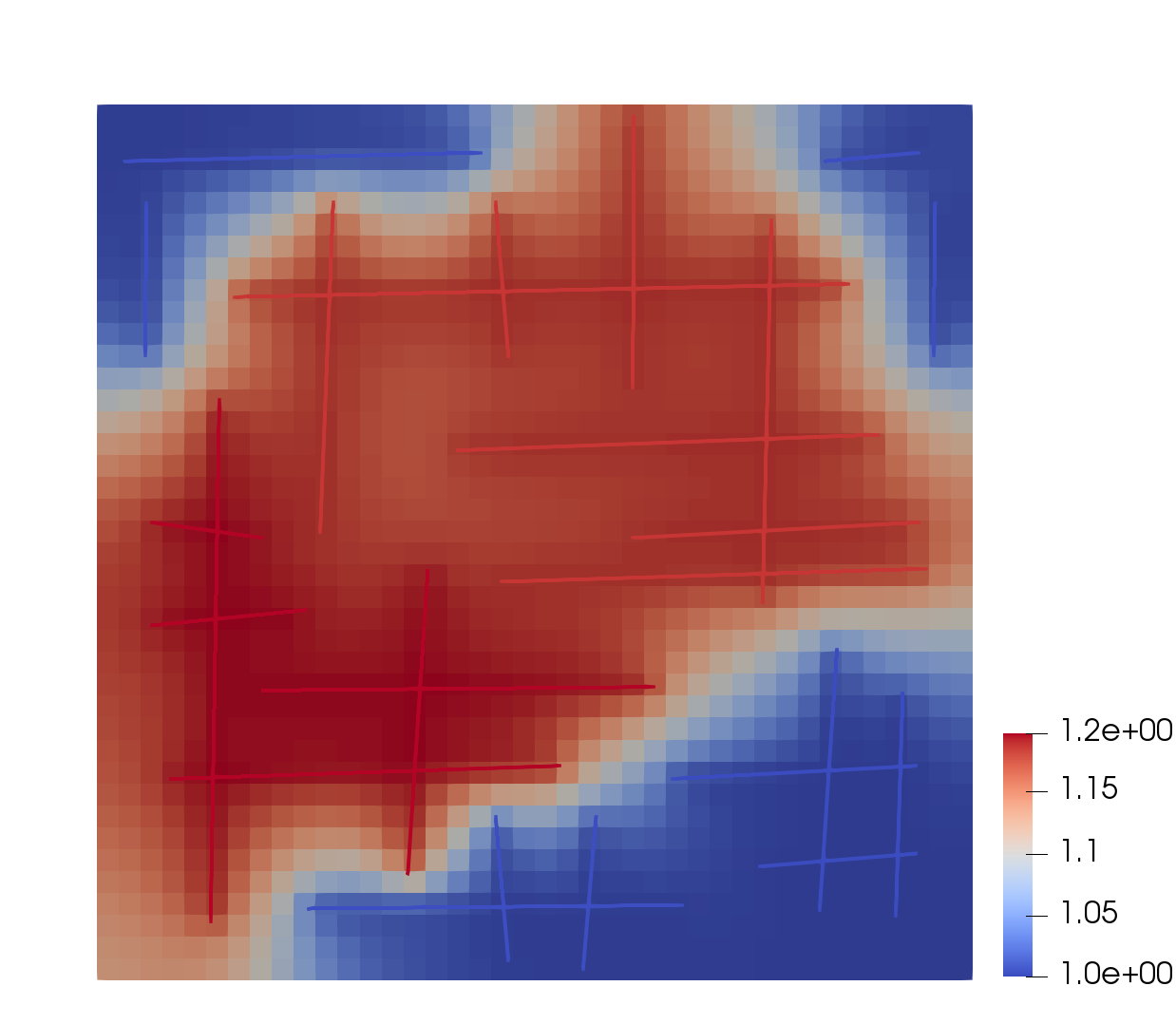}
\caption{Solution on the fine grid using the coupled scheme for two-continuum media (\textit{2C}). Solution $p^n$ for $n=10, 30$ and $50$ (from left to right)}
\label{fig:2c-uc}
\end{figure}

\begin{figure}[h!]
\centering
$20 \times 20$ coarse grid\\
\includegraphics[width=0.32 \textwidth]{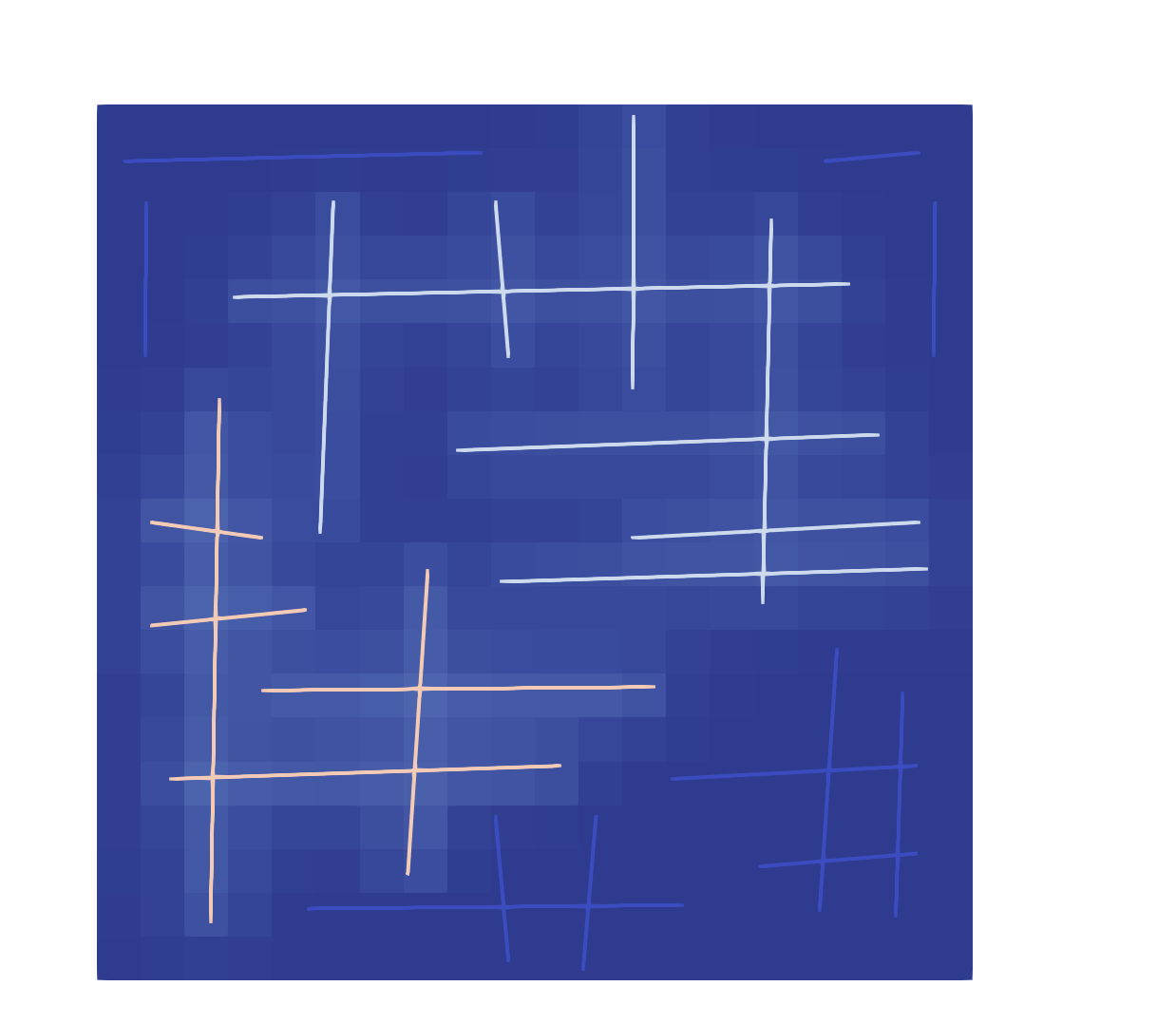}
\includegraphics[width=0.32 \textwidth]{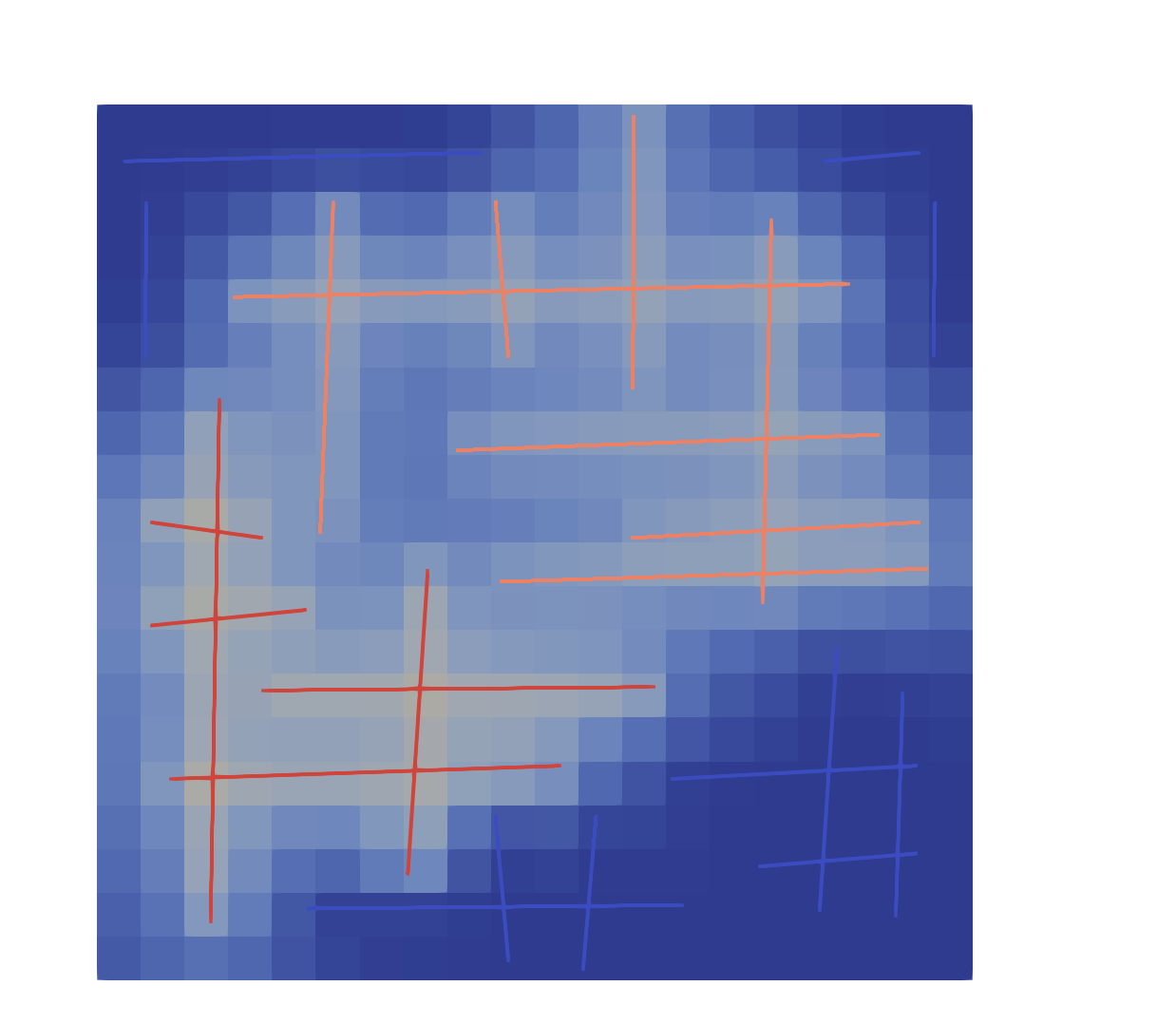}
\includegraphics[width=0.32 \textwidth]{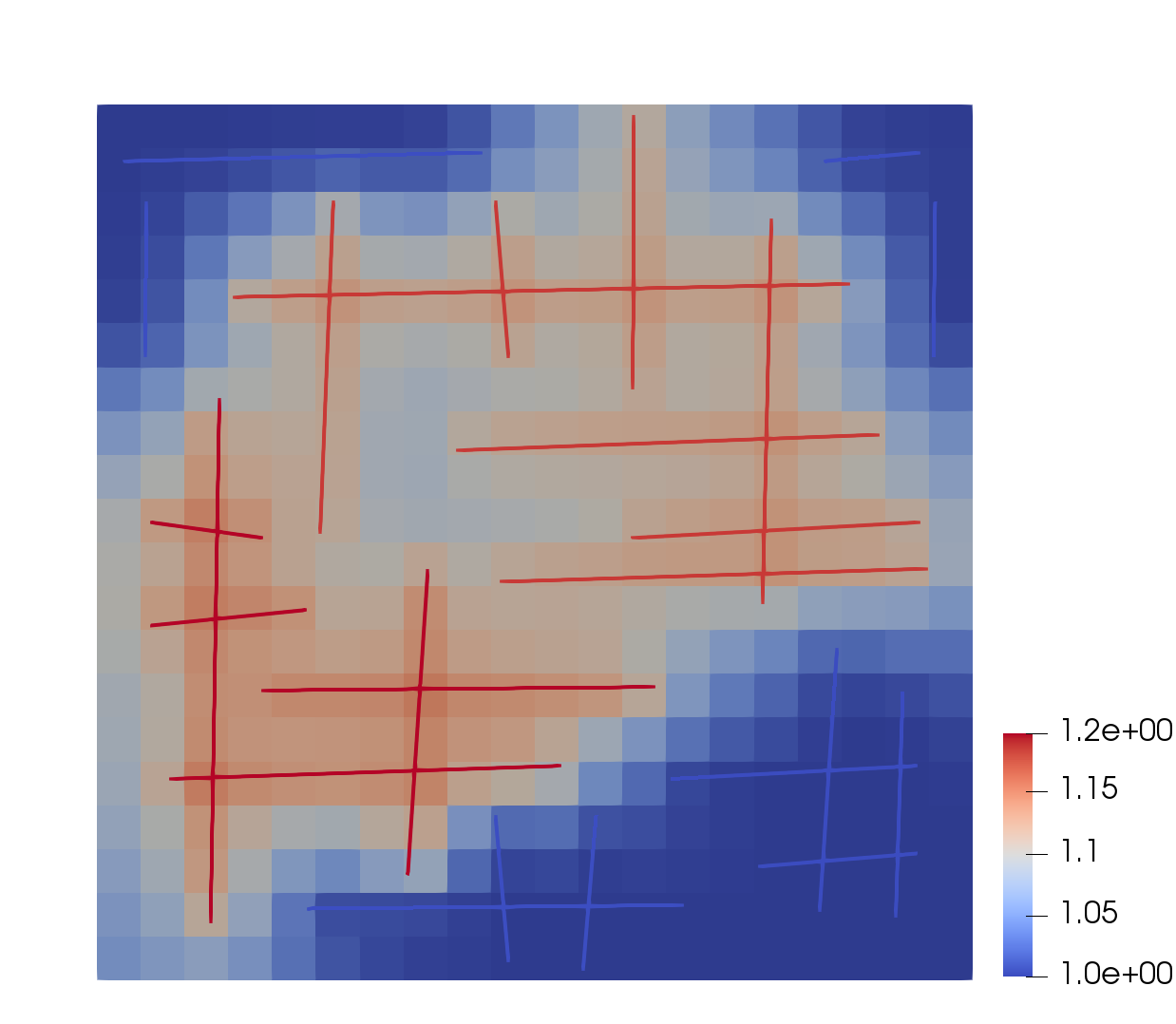}\\
\includegraphics[width=0.32 \textwidth]{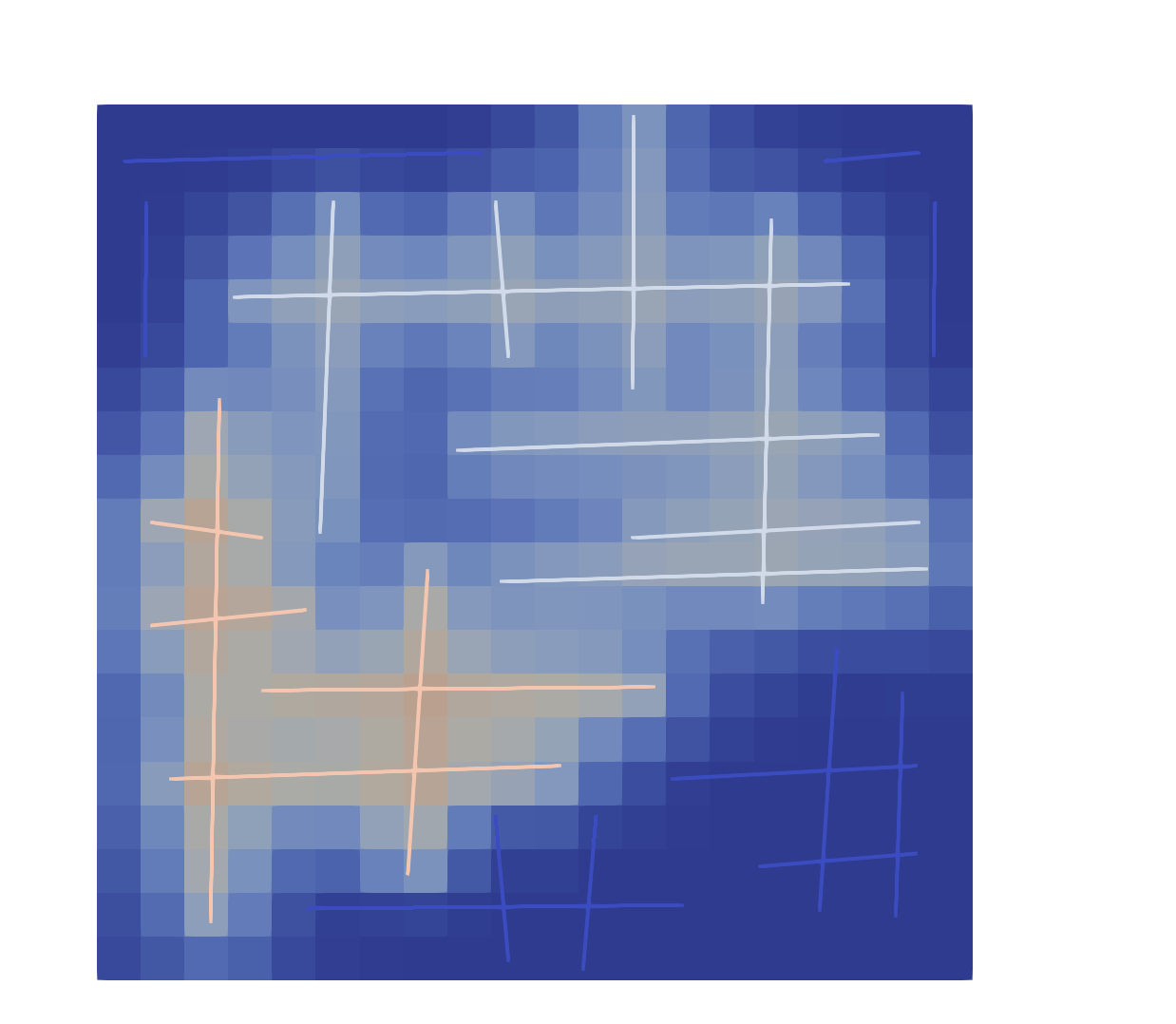}
\includegraphics[width=0.32 \textwidth]{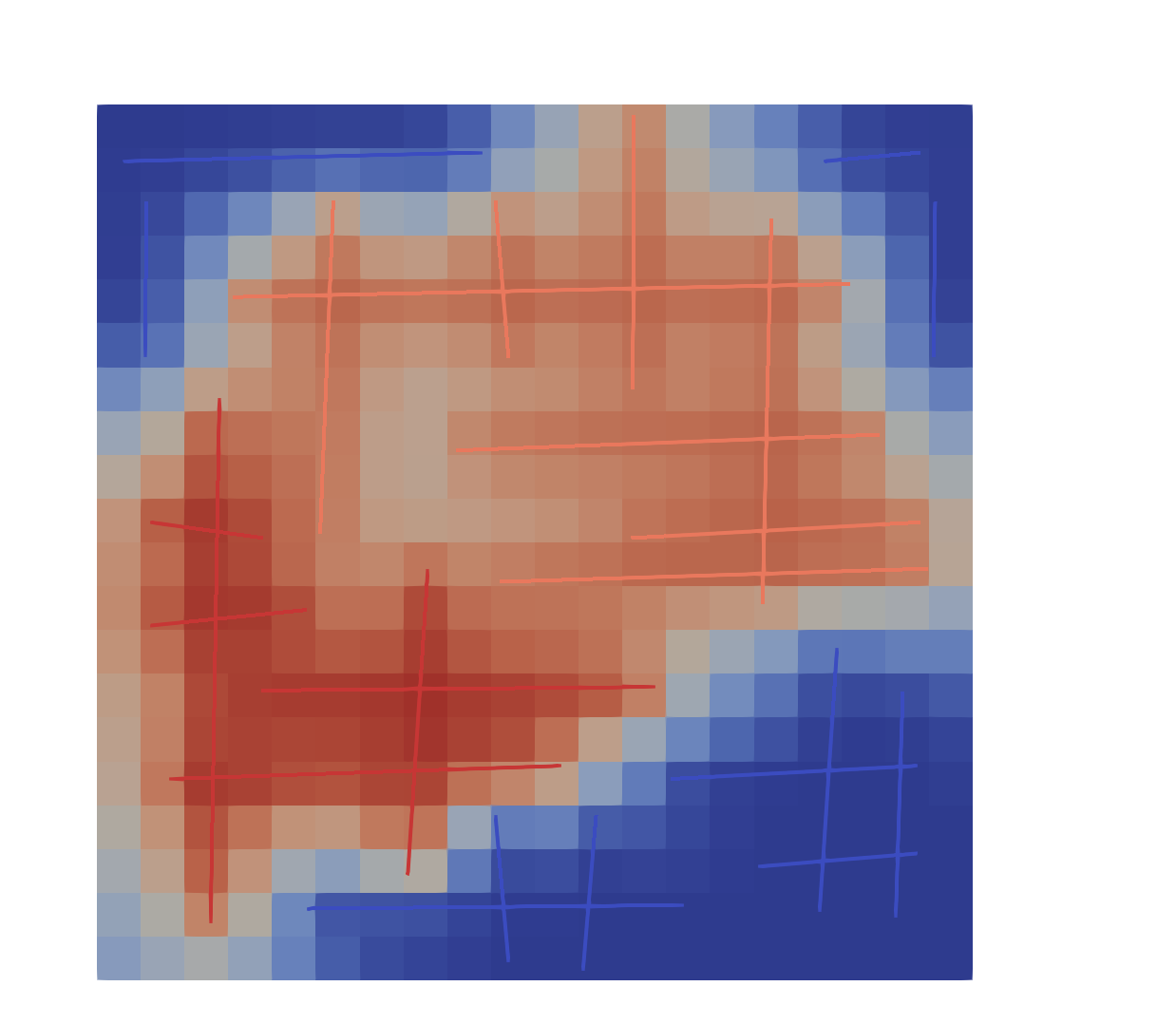}
\includegraphics[width=0.32 \textwidth]{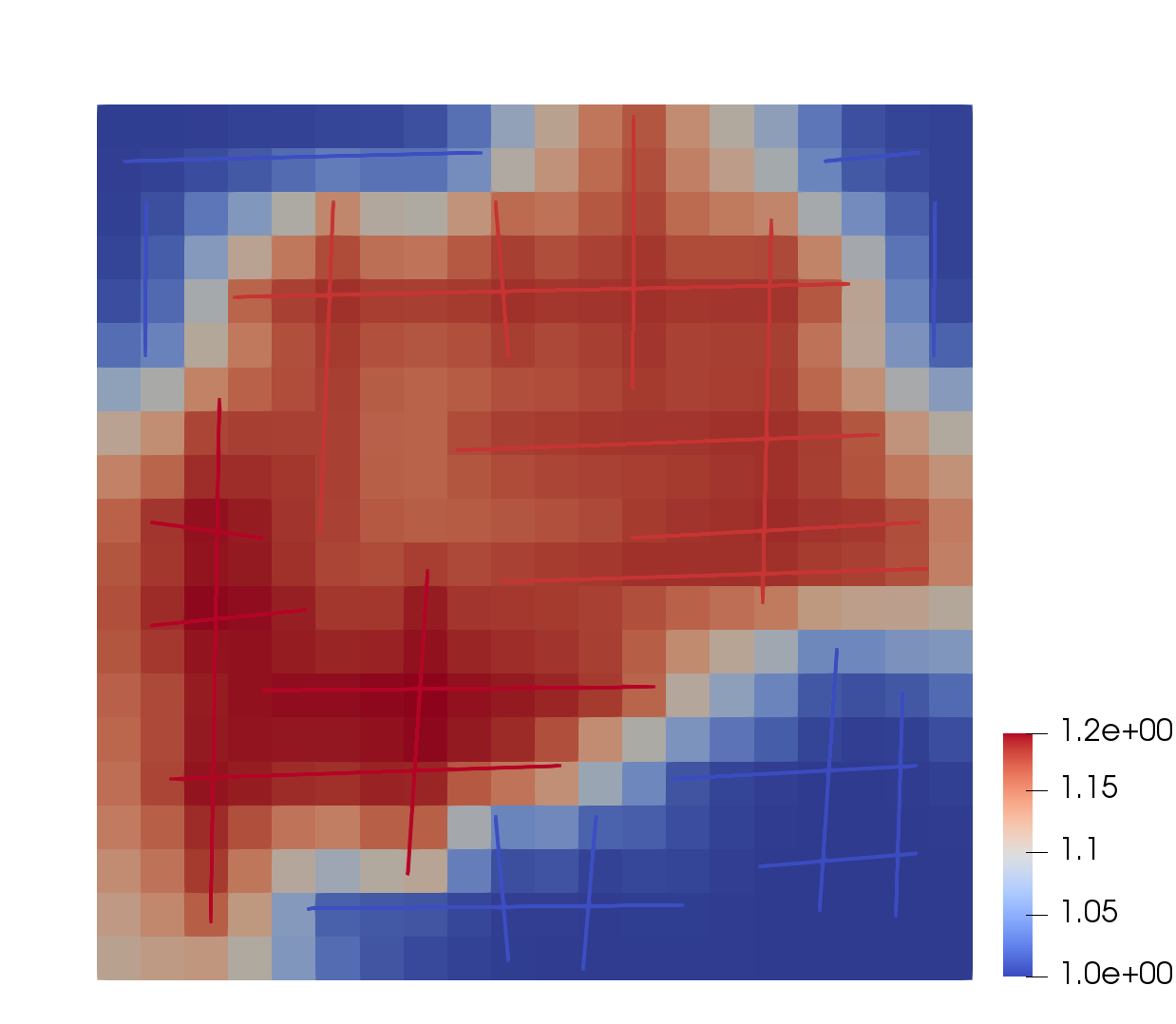}
$40 \times 40$ coarse grid\\
\includegraphics[width=0.32 \textwidth]{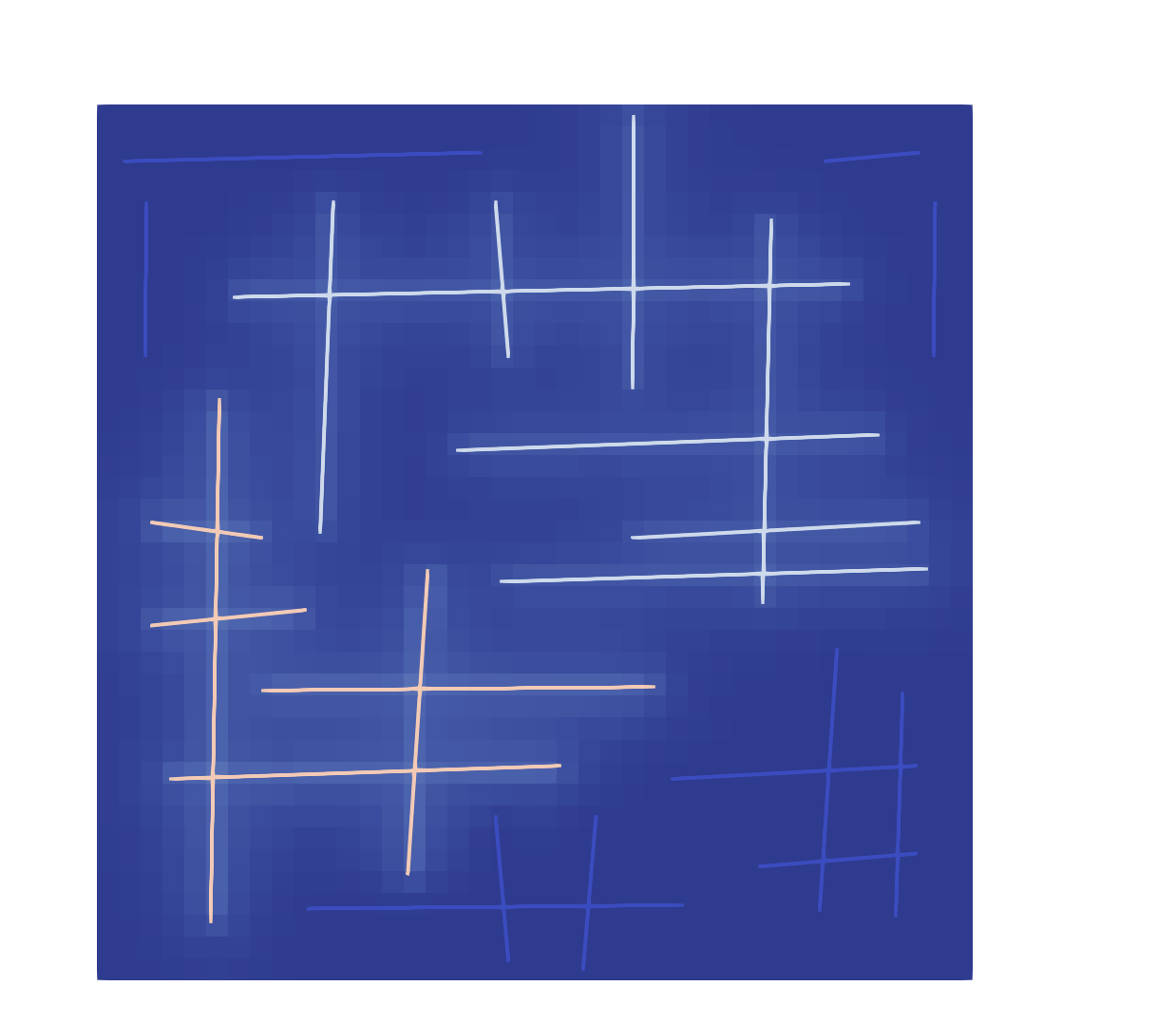}
\includegraphics[width=0.32 \textwidth]{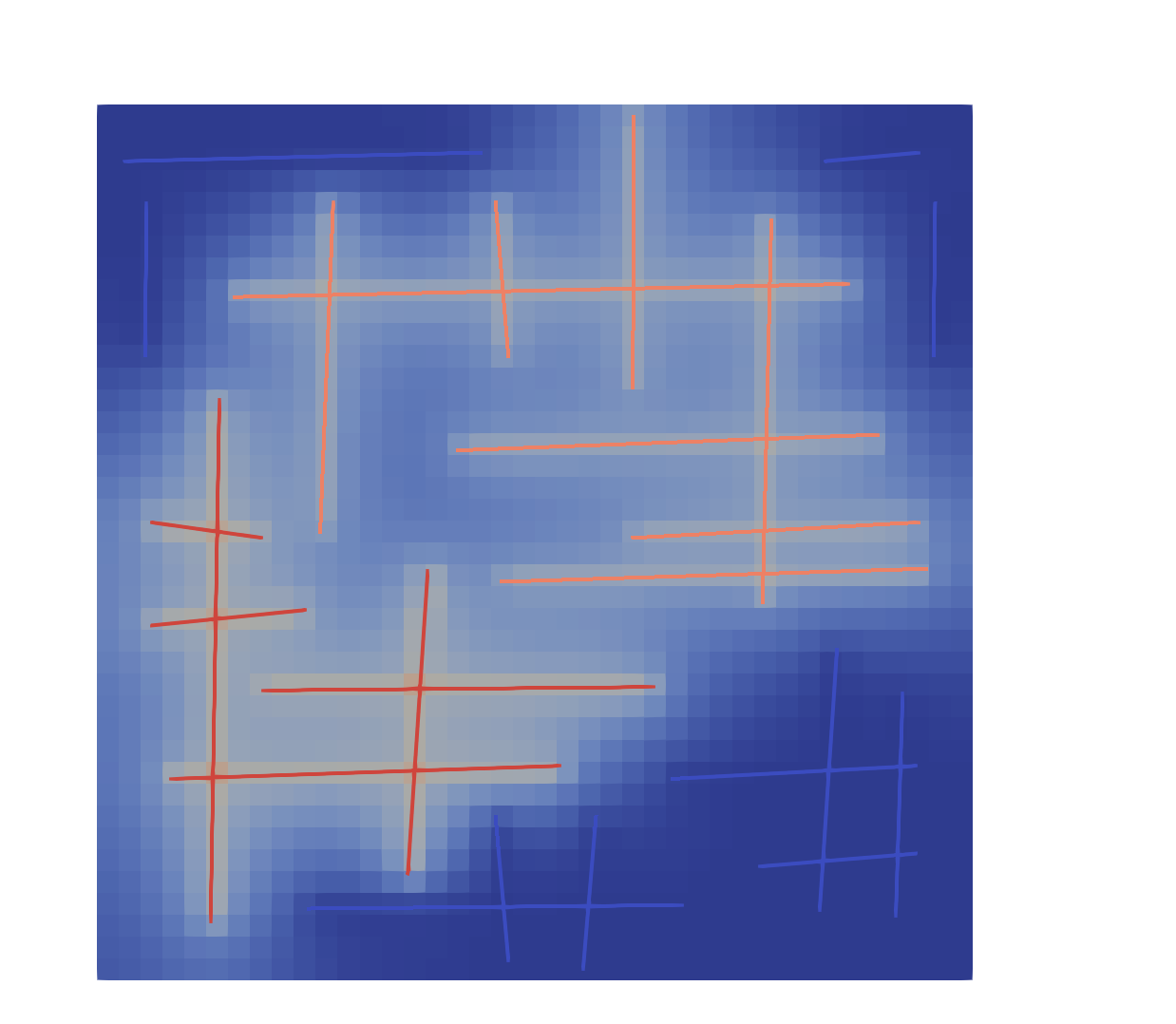}
\includegraphics[width=0.32 \textwidth]{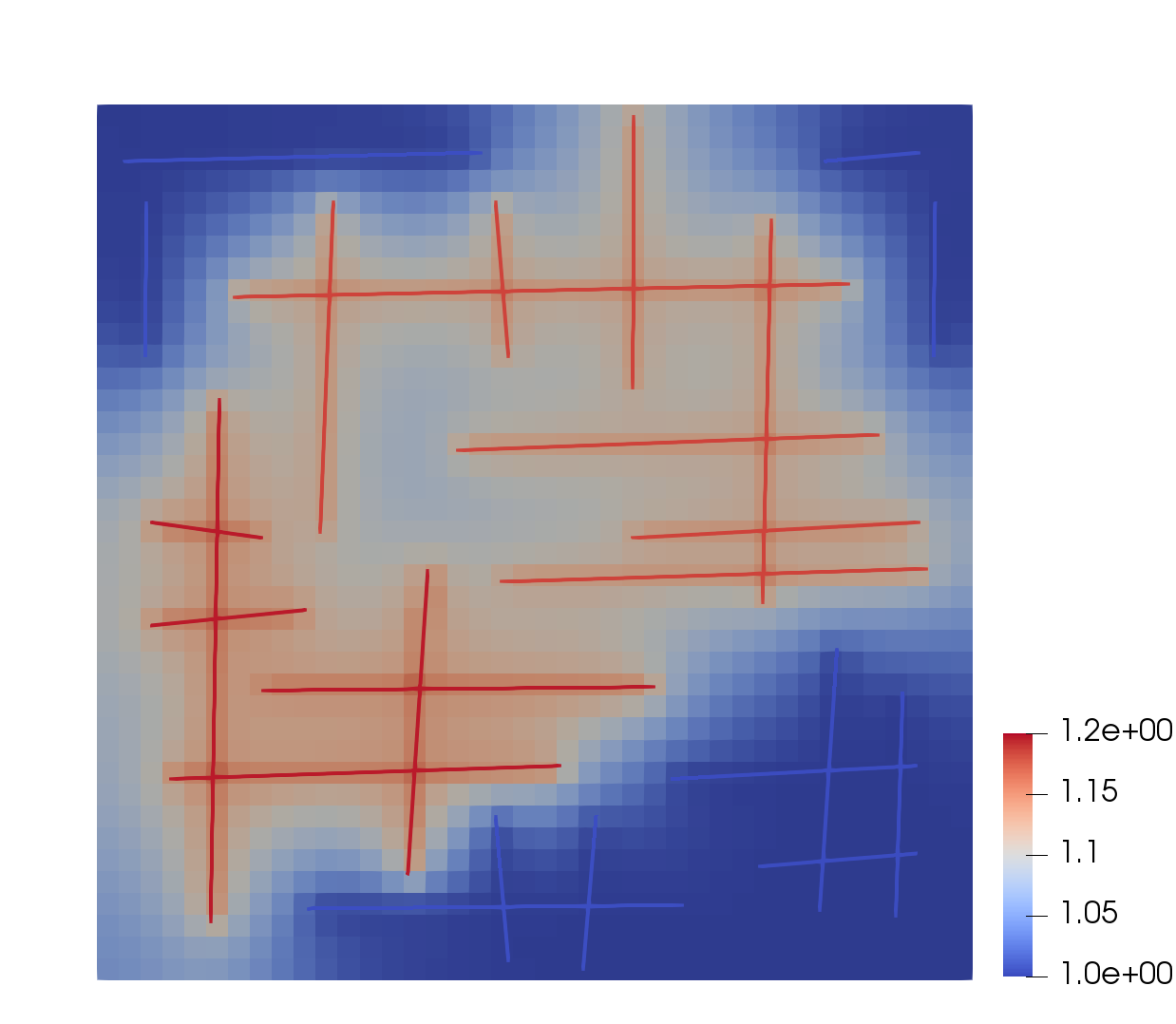}\\
\includegraphics[width=0.32 \textwidth]{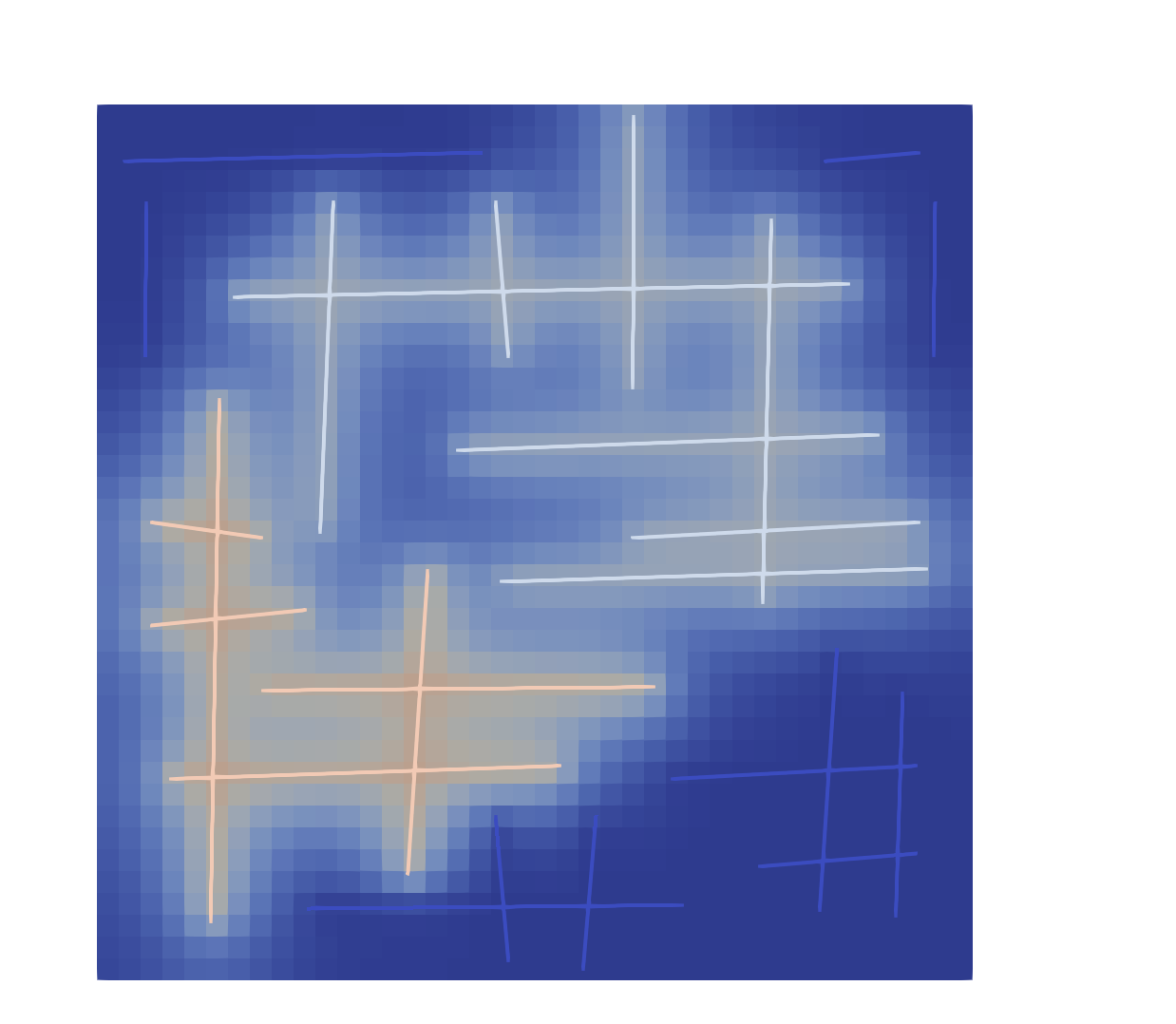}
\includegraphics[width=0.32 \textwidth]{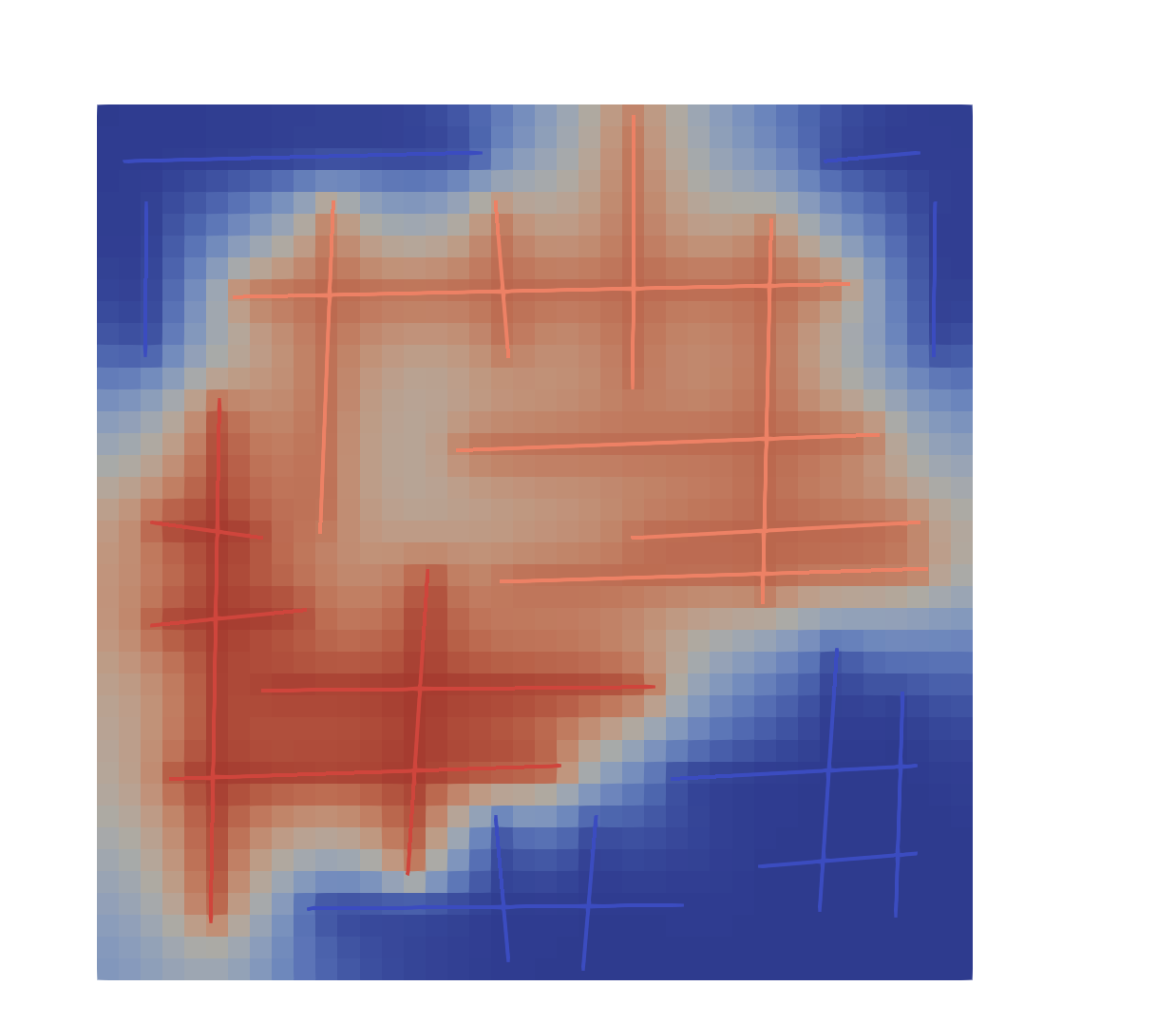}
\includegraphics[width=0.32 \textwidth]{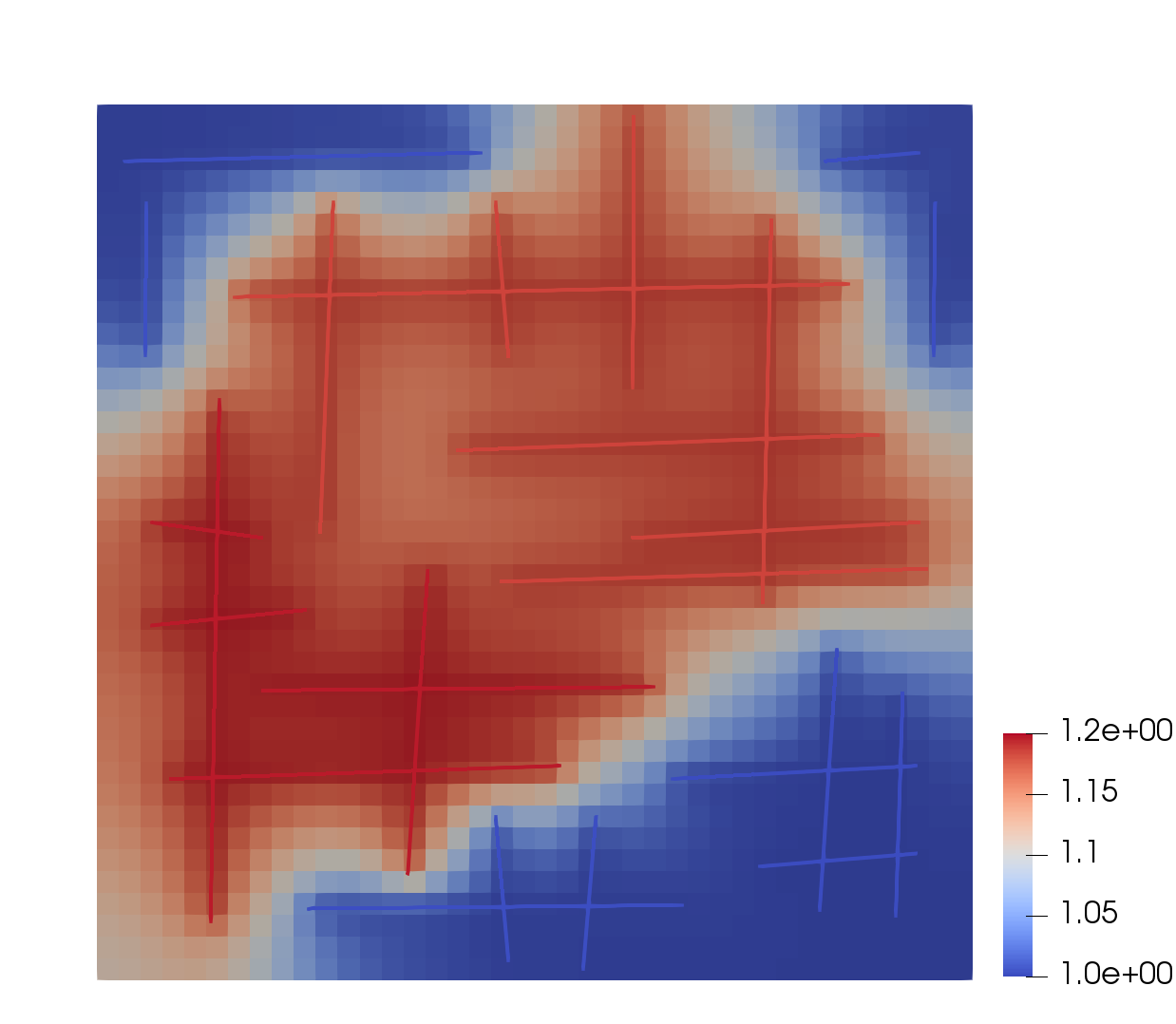}
\caption{Solution on the fine grid using the coupled scheme for three-continuum media (\textit{3C}). Solution $p^n$ for $n=10, 30$ and $50$ (from left to right). 
First row: porous matrix continuum. 
Second row: natural fracture continuum}
\label{fig:3c-uc}
\end{figure}

In Figure \ref{fig:2c-uc} and \ref{fig:3c-uc},  we present solution for two- and three--continuum porous media on $20 \times 20$ and $40 \times 40$ coarse grids.  We depict solution at three time layers $n=10, 30$ and $50$. 
In coupled scheme, we solve system of equations that have $DOF_H = 562$ for two-continuum problem (\textit{2C}) and $DOF_h = 962$ for three-continuum problem (\textit{3C}) on $20 \times 20$ coarse grid.  
Solution time is 0.174 sec and 0.557 sec for \textit{2C} and \textit{3C}, respectively.  
On the $40 \times 40$ coarse grid, we have  $DOF_h = 1930$ and   $DOF_h = 3530$ for two- and three-continuum problem, respectively. 
Solution time is 1.12 sec and 4.14 sec for \textit{2C} and \textit{3C} on $40 \times 40$ coarse grid.  
Average number of iterations for solution of the linear system of equations at each time layer is  $\bar{N}_{it} = 10$ for \textit{2C} and  \textit{3C} on $20 \times 20$ grid.  On $40 \times 40$ coarse grid, we have  $\bar{N}_{it} = 18$ for \textit{2C} and  \textit{3C}.  
The solution time is $265$ and $41$ times faster then fine grid solution for \textit{2C} model on $20 \times 20$ and $40 \times 40$ coarse grids, respectively. 
For \textit{3C} model, we obtain $190$ and $25$ times faster solution on $20 \times 20$ and $40 \times 40$ coarse grids, respectively.  
The solution using the NLMC coarse grid approximation is very accurate with $0.01$ \% of an error on the coarse grid for the coupled scheme.

\begin{figure}[h!]
\centering
\includegraphics[width=0.49 \textwidth]{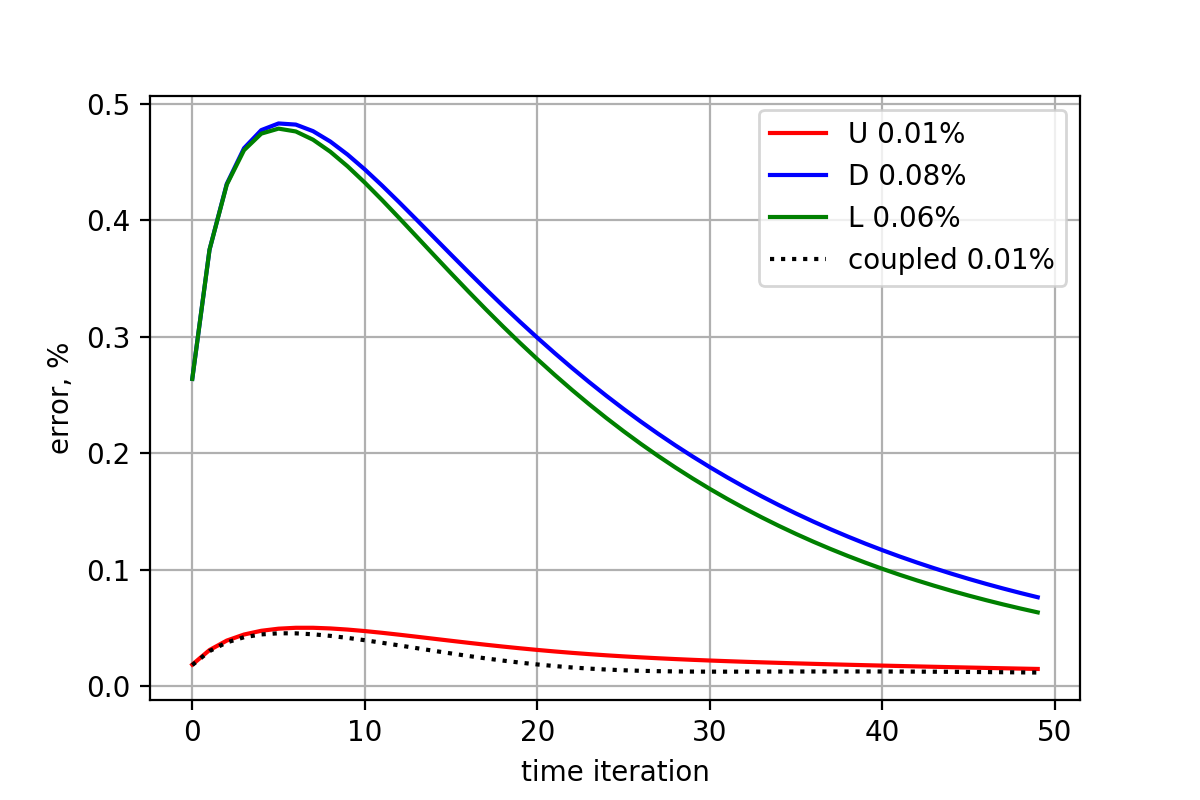}
\includegraphics[width=0.49 \textwidth]{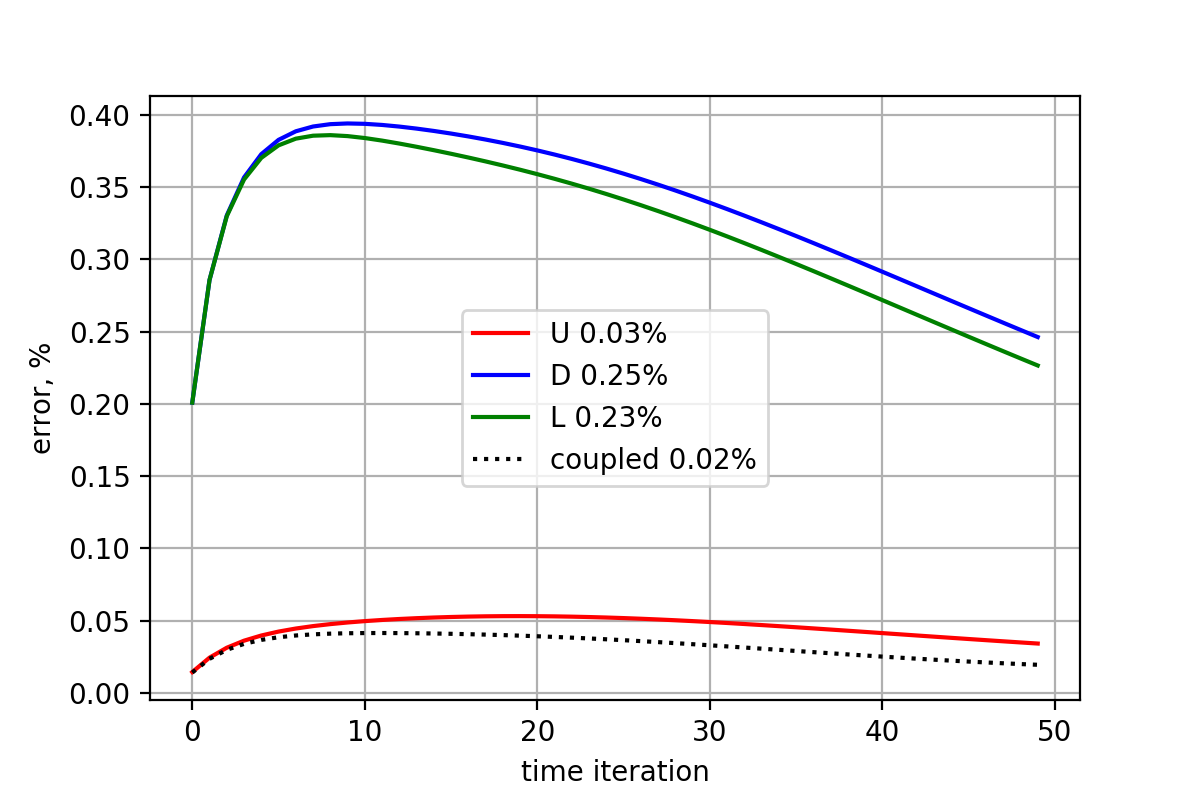}
\caption{Dynamic of the error (in percentage) for decoupled schemes with error at the final time. Coarse grid approximation using the NLMC method on $20 \times 20$ coarse grid. 
Left: two-continuum media (\textit{2C}).
Right: three-continuum media (\textit{3C})}
\label{fig:ms20}
\end{figure}

\begin{figure}[h!]
\centering
\includegraphics[width=0.49 \textwidth]{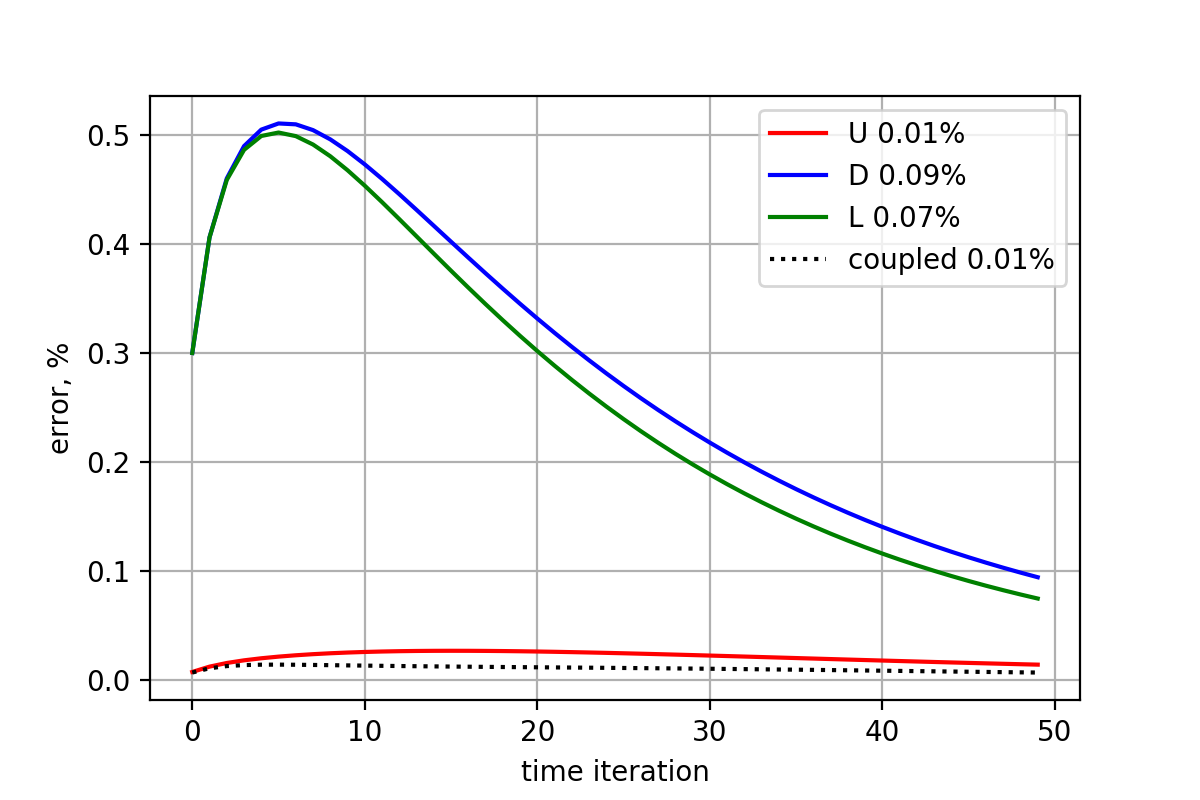}
\includegraphics[width=0.49 \textwidth]{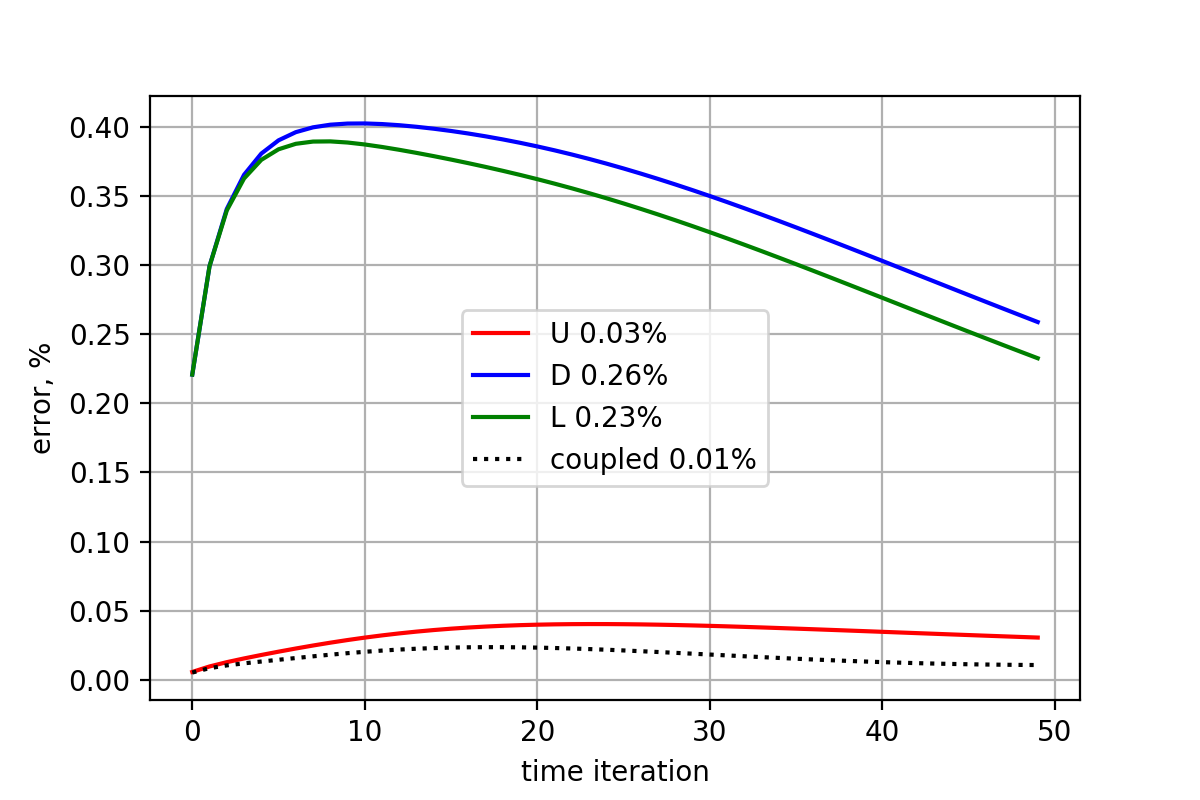}
\caption{Dynamic of the error (in percentage) for decoupled schemes with error at the final time.  Coarse grid approximation using the NLMC method on $40 \times 40$ coarse grid. 
Left: two-continuum media (\textit{2C}).
Right: three-continuum media (\textit{3C})}
\label{fig:ms40}
\end{figure}

\begin{table}[h!]
\centering
\begin{tabular}{|c|c|c|cc|}
\hline
 \multicolumn{5}{|c|}{Two-continuum media (\textit{2C})} \\ 
\hline
& time$_{tot}$(sec) & $e_H$ (\%)
& time$_1$  ($\bar{N}_{it,1}$)  
& time$_2$  ($\bar{N}_{it,2}$)   \\
\hline
\multicolumn{5}{|c|}{$20 \times 20$ coarse grid}\\
\hline
\textit{Coupled}	& 0.174 & 0.01 \% & \multicolumn{2}{|c|}{0.174 (10.0)}  	\\
\textit{L-scheme} 	& 0.035 & 0.06 \% & 0.021 (2.0) & 0.013 (10.0) \\
\textit{D-scheme} 	& 0.034 & 0.08 \% & 0.021 (2.0) & 0.012 (10.0) \\
\textit{U-scheme} 	& 0.041 & 0.01 \% & 0.025 (2.0) & 0.015 (10.0) \\
\hline
\multicolumn{5}{|c|}{$40 \times 40$ coarse grid} \\
\hline
Coupled	& 1.120 & 0.01 \% & \multicolumn{2}{|c|}{1.120 (18.04)}  	\\
\textit{L-scheme} 	& 0.135 & 0.07 \% & 0.105 (2.0) & 0.030 (18.08) \\
\textit{D-scheme} 	& 0.140 & 0.09 \% & 0.108 (2.0) & 0.031 (18.26) \\
\textit{U-scheme} 	& 0.151 & 0.01 \% & 0.121 (2.0) & 0.030 (18.08) \\
\hline
\end{tabular}
\caption{Time of the solution and the average number of iterations for two-continuum media (\textit{2C}) on the coarse grid using the NLMC method.  Coupled and decoupled schemes}
\label{tab-2c-ms}
\end{table}

\begin{table}[h!]
\centering
\begin{tabular}{|c|c|c|ccc|}
\hline
 \multicolumn{6}{|c|}{Three-continuum media (\textit{3C})} \\ 
\hline
& time$_{tot}$(sec) & $e_H$ (\%)
& time$_1$  ($\bar{N}_{it,1}$)  
& time$_2$  ($\bar{N}_{it,2}$)  
& time$_3$  ($\bar{N}_{it,3}$)   \\
\hline
\multicolumn{6}{|c|}{$20 \times 20$ coarse grid} \\
\hline
\textit{Coupled}	& 0.557 & 0.02 \% & \multicolumn{3}{|c|}{0.557 (10.0)}  \\
\textit{L-scheme} 	& 0.048 & 0.23 \%	
& 0.015 (1.0) & 0.020 (2.0) & 0.013 (10.0) \\
\textit{D-scheme} 	& 0.049 	& 0.25 \% 
& 0.015 (1.0) & 0.020 (2.0) & 0.013 (10.0) \\
\textit{U-scheme} 	& 0.055 	& 0.03 \%
& 0.016 (1.0) & 0.023 (2.0) & 0.014 (10.0) \\
\hline
\multicolumn{6}{|c|}{$40 \times 40$ coarse grid} \\
\hline
Coupled	& 4.14 & 0.01 \% & \multicolumn{3}{|c|}{4.14 (18.0)}  \\
\textit{L-scheme} 	& 0.234 & 0.23 \%	
& 0.064 (1.0) & 0.115 (2.0) & 0.032 (18.2) \\
\textit{D-scheme} 	& 0.204 	& 0.26 \% 
& 0.063 (1.0) & 0.111 (2.0) & 0.030 (18.2) \\
\textit{U-scheme} 	& 0.234 	& 0.03 \%
& 0.073 (1.0) & 0.128 (2.0) & 0.032 (18.3) \\
\hline
\end{tabular}
\caption{Time of the solution and the average number of iterations for three-continuum media (\textit{3C}) on the coarse grid using the NLMC method.   Coupled and decoupled schemes}
\label{tab-3c-ms}
\end{table}

In Figure \ref{fig:ms20} and \ref{fig:ms40}, we present dynamic of the relative errors for coupled scheme and three decoupling schemes (\textit{L}, \textit{D} and \textit{U}-schemes) on the $20 \times 20$ and $40 \times 40$ coarse grids, respectively.   
Similarly to the fine grid results, we sort coarse grid continua in ascending order based on their permeability.   In  \textit{L-scheme}, we first solve a problem with lower permeability. In \textit{U-scheme}, we first solve a problem for the continuum related to the higher permeability.   
In decoupled schemes,  we solve the equation for each continuum separately. 
On the  $20 \times 20$ coarse grid,   we  have $DOF_{H,1} = 400$ and $DOF_{H,2} = 162$ for \textit{2C} model, and $DOF_{H,1} = DOF_{H,2} = 400$ and $DOF_{H,3} = 162$ for  \textit{3C} model. 
On the  $40 \times 40$ coarse grid,  we  have $DOF_{H,1} = 1600$ and $DOF_{H,2} = 330$ for \textit{2C} model, and $DOF_{H,1} = DOF_{H,2} = 1600$ and $DOF_{H,3} = 330$ for  \textit{3C} model. 
From Figures \ref{fig:ms20} and  \ref{fig:ms40}, we observe that the coarse grid approximation using the NLMC method provides very good results with small errors for both coupled and decoupled schemes. 
However,   we again observe that the  \textit{U-scheme} gives better results with almost the same errors as a coupled scheme.

In Tables {\ref{tab-2c-ms} and \ref{tab-3c-ms}, we present the solution time and the average number of iterations for linear solver at each time layer.  
We present the total time of solution on the $20 \times 20$ and $40 \times 40$ coarse grids with the relative error in percentage at the final time.  
Similarly to the fine grid results, we present solution time related to each continuum equation with an average number of iterations for decoupled schemes. 
By system decoupling, we obtain a separate equation for each continuum and observe that the number of the iteration in the less permeable domain is smaller than in the higher permeable continuum.  
Solution time is $0.03-0.04$ and $0.04-0.05$ sec on $20 \times 20$ coarse grid for \textit{2C}  and \textit{3C} models which is $4.3$  and $10$ times faster then coupled scheme on the coarse grid.
On $40 \times 40$ coarse grid,  solution time is $0.1$ and $0.2$ sec for \textit{2C}  and \textit{3C} models which is $7.4$  and $17.7$ times faster then coupled scheme. 
The error between reference solution and solution using the NLMC method is very small, and decoupled schemes work similarly to the regular finite volume method.

\section{Conclusion}

We presented efficient decoupled schemes for multicontinuum flow problems in fractured porous media. The presented approach is based on the additive representation of the operator with semi-implicit approximation by time to decoupled equations for each continuum. 
We developed, analyzed, and investigated three decoupled schemes for solving the classical multicontinuum problems in fractured porous media on fine grids with finite volume approximation by space.  
The presented results show that the decoupled schemes are stable with respect to the initial condition and right-hand side and provide an accurate solution. We observe that the order of the continuum in solution sequence is matter, where a more accurate solution can be obtained when we first calculate the solution for a higher permeable continuum (\textit{U-scheme}). We observe that the continuum decoupling schemes are very efficient on the fine grid and can provide faster simulations ($20-40$ times faster than the solution using the coupled scheme on $200 \times 200$ fine grid). 
We extend the continuum decoupling approach for multiscale multicontinuum problems with nonlocal multicontinuum (NLMC) approximation on the coarse grid. We observe the same efficiency of the presented method with a very small error. 
Numerical results were performed for two- and three-continuum models in the two-dimensional formulation. 
By combining two techniques (NLMC and continuum decoupling), the simulation time becomes $0.04$ sec for the two-continuum model and $0.05$ sec for the three-continuum model for \textit{U-scheme} on the $20 \times 20$ coarse grid, where on the  $200 \times 200$ fine grid coupled scheme take $46.3$ sec and $105.6$ sec for simulation for two- and three-continuum model, respectively.

\bibliographystyle{plain}
\bibliography{lit}

\end{document}